\documentclass[a4paper,12pt,reqno]{amsart}
\usepackage{latexsym}
\usepackage{amssymb,amsfonts,amsmath,mathrsfs}
\newtheorem{theorem}{Theorem}[section]
\newtheorem{proposition}{Proposition}[section]
\newtheorem{lemma}{Lemma}[section]

\author{Micah B. Milinovich \and Nathan Ng}
\address{Department of Mathematics, University of Mississippi, University, MS 38677 USA}
\email{mbmilino@olemiss.edu}
\thanks{The first author is supported in part by an AMS-Simons travel grant and the NSA Young Investigator grant H98230-13-1-0217.  The second author is supported in part by an NSERC Discovery grant.}
\address{Department of Mathematics and
Computer Science, University of Lethbridge, Lethbridge, AB Canada T1K 3M4}
\email{nathan.ng@uleth.ca}


\title[Simple zeros of modular $L$-functions]{Simple zeros of modular $L$-functions}
\begin{document}
\begin{abstract}
Let $f$ be a primitive holomorphic cusp form of weight $k$, level $q$, and character $\chi$, and let $L(s,f)$ be its associated $L$-function. Assuming the generalized Riemann hypothesis for $L(s,f)$, we prove that the number of simple zeros $\rho_{\!_f}=\frac{1}{2}+i\gamma_{\!_f}$ of $L(s,f)$ satisfying $0<\gamma_{\!_f}\leq T$ is greater than a positive constant times $T(\log T)^{-\varepsilon}$ for any $\varepsilon >0$ and $T$ sufficiently large.
\end{abstract}
\maketitle

\section{Introduction}

Let $k$ and $q$ be positive integers, and let $\chi$ be a Dirichlet character modulo $q$. Further, let $S_k\big(\Gamma_0(q),\chi\big)$  denote the complex vector space of holomorphic modular forms of weight $k$ and character $\chi$ for the congruence subgroup $\Gamma_0(q)$, and let $H_k(q,\chi)$ denote the subset of $S_k\big(\Gamma_0(q),\chi\big)$ consisting of normalized holomorphic newforms. Attached to every $f \in H_k(q,\chi)$ is an $L$-function of the form
\begin{equation*}
\begin{split}
 \label{eq:eulerproduct}
L(s,f) \ \! &=  \ \! \sum_{n=1}^\infty \frac{\lambda_f(n)}{n^s}
\end{split}
\end{equation*}
for $\Re(s)>1$ where  $\lambda_f(n)$ is a multiplicative function. Here, we are normalizing so that $\lambda_f(1)=1$, the critical line of $L(s,f)$ is $\Re(s) =\frac{1}{2}$, and so that Deligne's bound gives $|\lambda_f(n)| \le d(n)$ where $d(n)$ denotes the number of positive divisors of a natural number $n$. In addition to this bound on its coefficients, the function $L(s,f)$ extends to an entire function, has an Euler product, and a functional equation relating $s$ to $1-s$.

\smallskip

The function $L(s,f)$ has  an infinite number of non-trivial zeros in the strip $0 < \Re(s) < 1$. We let $N_{\!_f}(T)$ denote the number of non-trivial zeros $\rho_{\!_f}=\beta_{\!_f}+i\gamma_{\!_f}$ satisfying $0\leq\beta_{\!_f}\leq1$ and $0\leq \gamma_{\!_f} \leq T$, counted with multiplicity. Then it is known that $ N_{\!_f}(T)  \sim  \frac{T}{\pi}  \log T$ as $T\to \infty$.  It is conjectured that, apart from a possible multiple zero at the central point  $s=\frac{1}{2}$, all the non-trivial zeros of $L(s,f)$ are simple. We let $N_{\!_f}^{s}(T)$ denote the number of simple non-trivial zeros of $L(s,f)$ with $0<\gamma_{\!_f} \leq T$. Booker \cite{B} has recently proved that $N_{\!_f}^{s}(T)$ is unbounded. Assuming the generalized Riemann hypothesis, which states that $\beta_{\!_{f}}=\frac{1}{2}$ for each non-trivial zero of $\rho_{\!_{f}}$ of $L(s,f)$, we prove the following quantitative lower bound for the number of simple zeros of  $L(s,f)$.

\begin{theorem} \label{simple} Let  $f\in H_k(q,\chi)$ and assume the generalized Riemann hypothesis for $L(s,f)$. 
 Then, for every $\varepsilon>0$ and sufficiently large $T$, the inequality
  $$N_{\!_f}^{s}(T) \gg  T (\log T)^{-\varepsilon}   $$ 
 holds where the implied constant depends on $f$ and $\varepsilon$.
\end{theorem}

 \smallskip

The  generalized Riemann hypothesis and the conjecture that the non-real zeros of $L(s,f)$ are all simple are just two of many important open problems concerning the distribution of its non-trivial zeros.  For instance, it has also been conjectured that the the multiset of positive imaginary parts of the zeros of $L(s,f)$ for all $f \in H_k(q,\chi)$ are linearly independent over the rationals. In particular, this linear independence conjecture implies the previous conjecture on the simplicity of the non-real zeros of $L(s,f)$.  Moreover, it is believed that, after normalizing, the level spacing distribution between the ordinates of the non-trivial zeros of $L(s,f)$ should agree with the spectral statistics from the Gaussian Unitary Ensemble (GUE). Among other things, the GUE conjecture  implies that almost all zeros of $L(s,f)$ are simple.

 \smallskip
 
Our main theorem provides some evidence for the conjecture that the non-real zeros of $L(s,f)$ are all simple. Moreover, it seems likely that an analogue of Theorem \ref{simple} can be proved for $L$-functions attached to cuspidal Maass newforms. Prior to Booker's result mentioned above, there were only a few examples of degree two $L$-functions that were known to possess infinitely many simple zeros. The first example of such an $L$-function is contained in  an article of Conrey and Ghosh \cite{CG} where they show  that  the $L$-function associated to the Ramanujan $\Delta$-function posseses infinitely many simple zeros. Cho \cite{Cho1,Cho2} has modified their method and found a few examples of Maass form $L$-functions which also possess infinitely many simple zeros. In each of these cases, one must first verify that the $L$-function in question has at least one simple zero and, in all of the mentioned examples, this was done computationally. The methods of \cite{Cho1,Cho2,CG} actually show that if $L(s,f)$ has one simple non-real zero, then for every $\varepsilon >0$, 
$N_{\!_f}^{s}(T) \gg T^{\frac{1}{6}-\varepsilon}$ for arbitrarily large values of $T$. In his recent paper, Booker \cite{B} builds upon the ideas of Conrey and Ghosh and unconditionally proves that $L(s,f)$ possesses infinitely many simple zeros for any  $f \in H_k(q,\chi)$.   Though the argument  in \cite{B} does not provide a quantitative lower bound for $N_{\!_f}^s(T)$, one nice aspect of the proof is that it does not require that one first computationally verify that $L(s,f)$ has at least one simple zero. Therefore, the result in \cite{B}  can be applied to an infinite class of $L$-functions.

\subsection{Counting zeros of $L$-functions}

Several other approaches have been developed to count the number of zeros of an $L$-function in a given range with certain multiplicities. We note that none of these methods has led to a proof (even under the assumption of the generalized Riemann hypothesis) that infinitely many non-trivial zeros of $L(s,f)$ are simple for an arbitrary holomorphic newform $f \in H_k(q,\chi)$.

\smallskip

Selberg's method \cite{Se1} shows that a positive proportion of the non-trivial zeros of the Riemann zeta-function, $\zeta(s)$, are on the critical line. This is accomplished by counting sign changes of Hardy's function $Z(t)$ and every zero detected by this method has odd multiplicity.  Hafner has adapted Selberg's method to the $L$-functions attached to Hecke eigenforms of the full modular group and to Maass wave forms, and proved that a positive proportion of the zeros of these $L$-functions are on the critical line have odd multiplicity (see \cite{H1,H2,H3}). More recently, Rezvyakova \cite{Rez} has extended the techniques of Hafner and Selberg to the $L$-functions attached to any $f \in H_k(q,\chi)$.

\smallskip

In Montgomery's important article on the pair correlation of the zeros of the Riemann zeta-function \cite{Mo1}, a conditional approach to detecting simple zeros of $\zeta(s)$ was introduced. He showed that the Riemann hypothesis implies that two-thirds of the non-trivial zeros of $\zeta(s)$ are simple. This proportion  can be improved slightly (see \cite{ChG}), and these techniques can also be applied to the family of Dirichlet $L$-functions, counting zeros in $q$-aspect (see \cite{O} and \cite{CLLR}).  Moreover, Montgomery's pair correlation conjecture implies that almost all zeros of $ \zeta(s)$ are simple. Rudnick and Sarnak \cite{RS} generalized Montgomery's results to $n$-level correlations of zeros of principal automorphic $L$-functions.  Using the results in  \cite{RS}, it is possible to show that  a positive proportion of the zeros of $L(s,f)$ for $f \in H_k(q,\chi)$ have multiplicity less than or equal to two under the assumption of the generalized Riemann hypothesis for $L(s,f)$.

\smallskip
 
The Riemann zeta-function was the first example of an $L$-function known to have a simple nontrivial zero. In 1974, Levinson \cite{L} introduced a new method which established at least one-third of the zeros of $\zeta(s)$ lie on the critical line. Heath-Brown  \cite{HB} and Selberg (unpublished) independently observed that the method actually detects simple zeros. Bauer \cite{Ba} later showed that Levinson's method works in the case of Dirichlet $L$-functions. Farmer \cite{F} adapted Levinson's method to the case of $L$-functions attached to Hecke eigenforms for the full modular group. He was able to show that a positive proportion of its zeros of these $L$-functions have multiplicity less than or equal to $3$. Recently, Conrey, Iwaniec, and Soundararajan \cite{CIS} combined their asymptotic large sieve with Levinson's method to obtain lower bounds for the proportion of simple zeros on the critical line for the family of the twists by primitive Dirichlet characters of a fixed automorphic $L$-function of degree 1, 2, or 3.

\smallskip

In a series of paper beginning in the 1980's, Conrey, Ghosh, and Gonek developed a technique to prove the existence of simple zeros of $L$-functions using the non-vanishing of their derivatives. Given an $L$-function $L(s)$ with non-trivial zeros $\rho$, they estimate mean-values of the form 
\[
 \sum_{0 < \Im(\rho) \leq T} \!\!\! L'(\rho)
 \]
and related sums. Indeed, the argument of Conrey and Ghosh in \cite{CG} is a variant of this.  In \cite{CGG3} and \cite{CGG1}, Conrey, Ghosh, and Gonek apply this idea to $\zeta(s)$ and to $\zeta_K(s)$, the Dedekind zeta-function of a quadratic number field $K$. These three authors, and later Bui and Heath-Brown, proved that positive proportion results for the number of simple zeros can be obtained for these zeta-functions if additional hypotheses such as the Riemann hypothesis and generalized Lindel\"{o}f hypothesis are assumed (see \cite{BHB}, \cite{CGG2}, and \cite{CGG4}). 

\smallskip

Finally, we remark that there are a number of results which detect simple zeros in families of $L$-functions at the central point, and in some cases these results have important arithmetic applications. For instance, it is known that for a fixed $L$-function $L(s,f)$ with $f\in H_k(q,\chi_0)$, there are infinitely many twists by quadratic characters $\chi_{D}$ such that $L(s,f\otimes\chi_{D})$ has a simple zero at $s=\frac{1}{2}$. Here $\chi_0$ is the principal character (mod $q$). This was proven, independently, by Murty and Murty \cite{MM} when $q=2$ and by Bump, Friedberg, and Hoffstein \cite{BFH} for all $q$. This analytic condition is required in Kolyvagin's work on finiteness of the Mordell-Weil and Tate-Shafarevich groups of an elliptic $E$ over $\mathbb{Q}$ when $L(\frac{1}{2},E)\ne 0$, where $s=\frac{1}{2}$ is the central point.  Another important result, due to Kowalski and Michel \cite{KM}, is that at least 7/16 of the $L$-functions $L(s,f)$ for $f \in H_2(q,\chi_0)$ with $q$ prime and $\varepsilon_{_f}=-1$  have a simple zero at $s=\frac{1}{2}$. This result can be used to give a lower bound for the rank of $J_0(q)$, the Jacobian of the modular curve $X_0(q)$.

 \smallskip

\subsection{Auxiliary Theorems}  We now state two key mean-value estimates for the derivative of $L(s,f)$ which are used to deduce Theorem \ref{simple}.

\begin{theorem}\label{2nd_moment}
Let  $f\in H_k(q,\chi)$ and assume the generalized Riemann hypothesis for $L(s,f)$. Then
$$ \big(A_{\!_f} + o(1)\big) \ \! T \log^4 \! \Big( \frac{\sqrt{q} \ \! T}{2 \pi}\Big) \leq \sum_{0<\gamma_{\!_f}\leq T} \big| L'(\rho_{\!_{f}},f) \big|^{2} \leq \big(B_{\!_f}+o(1) \big) \ \! T \log^4 \! \Big( \frac{\sqrt{q} \ \! T}{2 \pi}\Big) $$
when $T$ is sufficiently large where the $o(1)$ terms are $O( 1/\sqrt{ \log\log T \ \! }).$ Here
$$ A_{f}=\left( \frac{17-\sqrt{145}}{12\pi} \right) c_{f} \quad \text{and} \quad B_{f}=\left( \frac{17+\sqrt{145}}{12\pi} \right) c_{f}$$
where $c_{f}$ is a positive constant defined by
\begin{equation}
 \label{eq:cfdefinition}
  c_{f} = \frac{(4\pi)^k}{\Gamma(k)} \frac{||f||^2}{\textup{vol}(\Gamma_0(q)\backslash \mathfrak{h}) },
  \end{equation}
with $\mathfrak{h}$ denoting the Poincar\'{e} upper half-plane and the norm on $f$ defined by 
$$ ||f||^2 = \int_{\Gamma_0(q)\backslash \mathfrak{h}} |f(z)|^2   \ \! y^k \ \! \frac{dx \ \! dy}{y^2} $$
\end{theorem}

\noindent {\sc Remark.} The constant $c_f$ in the above theorem arises as follows:
$$ c_f = \lim_{x\to \infty} \frac{1}{x} \sum_{n\leq x} |\lambda_f(n)|^2.$$
The fact that $c_f>0$ exists, and is finite, essentially follows from the work of Rankin and Selberg (see Proposition \ref{dmva}, below). Equivalently, if we define the convolution $L$-function $$L(s,f\!\times \!\bar{f})=\sum_{n\ge 1} \frac{|\lambda_f(n)|^2}{ n^{s}}$$ for $\Re(s)>1$, then it can be shown that $L(s,f\!\times \!\bar{f})$ can be continued to all of $\mathbb{C}$ apart from a simple pole at $s=1$ with a residue of $c_f$. 

\smallskip

Observe that the lower bound in Theorem \ref{2nd_moment} implies that $L(s,f)$ possesses infinitely many simple zeros. In fact, the quantitative bound $N_{\!_f}^s(T) \gg_\varepsilon T^{1-\varepsilon}$ for any $\varepsilon>0$ follows from this lower bound and the generalized Lindel\"{o}f hypothesis which implies that $|L'(\rho_{\!_{f}},f)| \ll_\varepsilon |\gamma_{\!_f}|^\varepsilon$ for every non-trivial zero $\rho_{\!_{f}}$ of $L(s,f)$. In order to prove the sharper quantitative estimate for $N_{\!_f}^s(T)$ given in Theorem \ref{simple}, we require the more precise upper bounds for higher moments of $L'(\rho_{_f},f)$ given in the following theorem.

\begin{theorem}\label{2kth_moment}
Let  $f\in H_k(q,\chi)$ and assume the generalized Riemann hypothesis for $L(s,f)$. Then, for natural numbers $m$ and real numbers $\ell \ge \frac{1}{2}$ and $\varepsilon>0$, we have
\[
\sum_{0<\gamma_{\!_f}\leq T} \big| L^{(m)}(\rho_{\!_f},f) \big|^{2\ell} \ll T(\log T)^{\ell^2 +2 \ell m +1+\varepsilon}
\]
when $T$ is sufficiently large, where the implied constant depends on $f$, $m$, $\ell$, and $\varepsilon.$
\end{theorem}

The above theorem is proved using techniques developed by of Soundararajan \cite{S}, as modified by the first author \cite{M}. We actually use Soundararajan's techniques to prove the following theorem, and then deduce Theorem \ref{2kth_moment} via an application of Cauchy's integral formula (see \textsection 8).

\begin{theorem}\label{shifted}
Let  $f\in H_k(q,\chi)$ and assume the generalized Riemann hypothesis for $L(s,f)$. Let $w\in\mathbb{C}$ with $|w|\le 1$ and $|\Re(w)| \le (\log T)^{-1}$. Then, for every positive real number $\ell$ and arbitrary $\varepsilon>0$, the inequality
\[
\sum_{0<\gamma_{\!_f}\leq T} \big| L(\rho_{\!_f} \!+\! w,f) \big|^{2\ell} \ll T(\log T)^{\ell^2+1+\varepsilon}
\]
holds uniformly in $w$ for sufficiently large $T$ where the implied constant depends on $f$, $\ell$, and $\varepsilon.$
\end{theorem}

\smallskip

There are two aspects which make the proofs of Theorem \ref{2kth_moment} and Theorem \ref{shifted} different than the proofs of the analogous theorems in \cite{M}. The first difference is that Theorem \ref{2kth_moment} holds for all real numbers $\ell \ge \frac{1}{2}$ where as the corresponding result\footnote{Although, using the techniques in the present paper, the proof of Theorem 1.1 in \cite{M} can be modified to hold for all $\ell \ge \frac{1}{2}$.} in \cite{M} was proved only for $\ell \in \mathbb{N}$. The second difference is in how we estimate the frequency of large values of a certain sum of the Fourier coefficients $\lambda_f(n)$ supported on the squares of the primes when averaged over the zeros of $L(s,f)$. If, in addition to the generalized Riemann hypothesis for $L(s,f)$, we were willing to assume the generalized Riemann hypothesis for the Dirichlet $L$-function, $L(s,\chi)$, and for the symmetric square $L$-function of $f$, $L(s,\text{sym}^2 f)$, then the proof in \cite{M} would carry over in a fairly straightforward manner. Analogous assumptions were made in the work of Soundararajan and Young \cite{SY}.  In the present case, we could deduce from the generalized Riemann hypotheses for $L(s,\chi)$ and  $L(s,\text{sym}^2 f)$ that the estimate
\[
D(s)=\sum_{p \leq z} \frac{(\lambda_f(p^2)\!-\!\chi(p))}{p^{2s}} \frac{\log(z/p)}{\log z} \ll_f \log\log\log T
\]
holds for $\Re(s) \ge \frac{1}{2}$, $2\le z \le \sqrt{T}$, and $T\le \Im(s) \le 2T$ when $T$ is sufficiently large. A sum of this form arises from the prime square contribution to an inequality for $\log|L(s,f)|$ that we prove in \textsection 3.  In order to circumvent these assumptions on $L(s,\chi)$ and $L(s,\text{sym}^2 f)$, instead of bounding the Dirichlet polynomial $D(s)$ point-wise, we prove an upper estimate on the number of zeros $\rho_{_f}$ of $L(s,f)$ with $T\le \Im(\rho_{_f}) \le 2T$ for which $|D(\rho_{_f})|$ is large. This argument is similar to the analysis Soundararajan's work but appears to be new as a point-wise bound for analogous sums supported on the squares of primes was used in \cite{M}, \cite{S}, and \cite{SY}.

\subsection{The proof of Theorem \ref{simple}}

We now indicate how Theorems \ref{2nd_moment} and \ref{2kth_moment} can be combined to prove Theorem \ref{simple}, and thus obtain a  quantitative lower bound for the number of simple zeros of $L(s,f)$ in the strip $0<\Im(s) \le T$.

\begin{proof}[Proof of Theorem \ref{simple}] Let $\varepsilon>0$ be fixed. Since a non-trivial zero $\rho_{\!_{f}}$ of $L(s,f)$ is simple if and only if $L'(\rho_{\!_f},f)\neq 0$, we observe that
\[
N_{\!_f}^{s}(T) \ = \!\!\! \sum_{\substack{0<\gamma_{\!_f}\leq T \\ L'(\rho_{\!_f},f)\neq 0}} \!\!\! 1.
\]
Therefore, H\"{o}lder's inequality implies that
\begin{equation}\label{cauchysimple}
\left\{ \sum_{0<\gamma_{\!_f}\leq T} \big| L'(\rho_{\!_f},f) \big|^{2}\right\}^{\ell} 
\leq   \sum_{0<\gamma_{\!_f}\leq T} \big| L'(\rho_{\!_f},f) \big|^{2\ell} \cdot \left\{ N_{\!_f}^{s}(T)\right\}^{\ell-1} 
 \end{equation}
 for any $\ell>1$. The lower bound in Theorem \ref{2nd_moment} implies that the left-hand side of this inequality is $\gg T^\ell (\log T)^{4\ell}$ while the inequality in Theorem \ref{2kth_moment} implies that the right-hand side of \eqref{cauchysimple} is 
 \[
 \ll  \left\{ N_{\!_f}^{s}(T)\right\}^{\ell-1} \cdot T  (\log T)^{(\ell+1)^2+\eta}
 \]
 for any $\eta>0$. Using these bounds, after a little rearranging, it follows that
 \[
 N_{\!_f}^{s}(T) \gg T (\log T)^{-(\ell-1)-\eta/(\ell-1)}
 \]
 for any $\ell>1$ and $\eta>0$. Choosing $\ell=1+\varepsilon/2$ and $\eta=\varepsilon^2/4$, we deduce the theorem.
\end{proof}

\smallskip

\noindent{\sc Remark.} Using Harper's \cite{Harper} recent and rather remarkable modification of Soundararajan's ideas, it may be possible to prove versions of Theorem \ref{2kth_moment} and Theorem \ref{shifted} with $\varepsilon=0.$ It does not appear, however, that this would lead to an improvement of Theorem \ref{simple} using the above proof.

\subsection{The proof of Theorem \ref{2nd_moment}} The proof of Theorem \ref{2nd_moment} combines ideas of Ramachandra \cite{Ram} and the second author \cite{Ng}, relying in a fundamental way on a classical result of Rankin \cite{Ran} concerning the average size of $|\lambda_f(n)|$.  Some of the technical aspects of the proof are quite involved, but the underlying ideas are relatively easy to explain. For this reason, we now give the proof the theorem to indicate these ideas. Throughout this article, we frequently encounter the arithmetic functions
\begin{equation}
 \label{eq:defncoeffs}
\alpha_f(n)=-\lambda_f(n)\log n \quad \text{ and } \quad \beta_{f,x}(n)=-\lambda_{f}(n)\log(x^2/n)
\end{equation}
and the quantity
\begin{equation}
  \label{eq:X}
  X = \frac{\sqrt{q}T}{2 \pi},
\end{equation}  
which is essentially the square root of the analytic conductor of $L(s,f)$ when $\Re(s) =\tfrac{1}{2}$ and $T<\Im(s) \leq 2T$. These arise naturally in the approximate functional equation of $L'(s,f)$ which has the shape 
\begin{equation}
  \label{eq:afe0}
  L'(s,f) = \sum_{n \le X} \frac{\alpha_f(n)}{n^s}+  \psi_{\!_f}(s) \sum_{n \le X} \frac{ \beta_{\bar{f},X}(n) }{n^{1-s}} + \mathcal{E}(s,f)
\end{equation} 
where $\lambda_{\bar{f}}(n) = \overline{\lambda_f(n)}$ denotes the coefficient of the Dirichlet series for $L(s,\bar{f}) := \overline{L(\overline{s},f)}$ when $\Re(s)>1$, and $\mathcal{E}(s,f)$ is an error term which, in mean-square, is smaller on average then the first two terms on the right-hand side of the above expression. (See Lemma \ref{approx} for a precise definition of $\mathcal{E}(s,f)$.) In \textsection 6, we prove the following estimates.

\begin{proposition} \label{mainterms}
Let  $f\in H_k(q,\chi)$ and assume the generalized Riemann hypothesis for $L(s,f)$. Then, for $T\geq 10$, we have
\begin{align} \label{eq:mainterm1}
 & \sum_{T<\gamma_{\!_f}\leq 2T} \Bigg|\sum_{n\leq X} \frac{\alpha_f(n)}{n^{\rho_{\!_f}}} \Bigg|^2 =  \frac{5 }{24\pi} c_f  T  \log^4 X + O\Big( T (\log T)^{4-2\delta} \Big)
 \end{align}
 and
 \begin{align} \label{eq:mainterm2}
 & \sum_{T<\gamma_{\!_f}\leq 2T} \Bigg|\sum_{n\leq X} \frac{\beta_{f,X}(n)}{n^{\rho_{\!_f}}} \Bigg|^2 = \frac{29 }{24\pi} c_f  T  \log^4 X + O\Big( T (\log T)^{4-2\delta} \Big) 
\end{align}
where $c_f$ is given by \eqref{eq:cfdefinition} and $\delta >\frac{1}{18}$ is the constant appearing in Proposition \ref{sums}. 
\end{proposition}

\begin{proposition}
  \label{errordmv}
Let  $f\in H_k(q,\chi)$ and assume the generalized Riemann hypothesis for $L(s,f)$.  Then, for $T\geq 10$, we have
$$\sum_{T<\gamma_{\!_f}\leq 2T} \big|\mathcal{E}\big(\rho_{\!_f}, f) \big|^2 = O\left( \frac{T \log^4 T}{\log\log T}\right) $$
where $\mathcal{E}(s,f)$ is the function appearing in Lemma \ref{approx}.
\end{proposition}

\smallskip

We now deduce Theorem \ref{2nd_moment} from Propositions \ref{mainterms} and \ref{errordmv}.

\begin{proof}[Proof of Theorem \ref{2nd_moment}] We define
\begin{equation}
 \label{eq:AfBf}
 \mathcal{A}_f(T):= \sum_{T<\gamma_{\!_f}\leq 2T} \Bigg| \sum_{n\leq X} \frac{\alpha_f(n)}{n^{\rho_{_f}}}  \Bigg|^{2}, \quad \mathcal{B}_f(T):= \sum_{T<\gamma_{\!_f}\leq 2T} \Bigg|  \sum_{n\leq X} \frac{\beta_{\bar{f},X}(n)}{n^{1-\rho_{_f}}} \Bigg|^{2},
\end{equation}
and
\begin{equation}
  \label{eq:Ef}
\mathcal{E}_f(T):= \sum_{T<\gamma_{\!_f}\leq 2T} \big|\mathcal{E}\big(\rho_{\!_f}, f) \big|^2. 
\end{equation}
The generalized Riemann hypothesis for $L(s,f)$ implies that $1-\rho_{\!_f}=\overline{\rho_{\!_f}}$ for any non-trivial zero $\rho_{\!_f}$ of $L(s,f)$, so it follows from this assumption that
$$  \mathcal{B}_f(T)= \sum_{T<\gamma_{\!_f}\leq 2T} \Bigg|  \sum_{n\leq X} \frac{\beta_{\bar{f},X}(n)}{n^{1-\rho_{_f}}} \Bigg|^{2} = \sum_{T<\gamma_{\!_f}\leq 2T} \Bigg|  \sum_{n\leq X} \frac{\beta_{f,X}(n)}{n^{\rho_{_f}}} \Bigg|^{2}.$$
Moreover, since $|\psi_{\!_f}(\tfrac{1}{2}+it)|=1$ for real $t$,  the generalized Riemann hypothesis for $L(s,f)$ also implies that $|\psi_{\!_f}(\rho_{\!_f}) |=1$ for all non-trivial zeros $\rho_{\!_f}$ of $L(s,f)$. Therefore, using \eqref{eq:afe0}, expanding, and then applying the Cauchy-Schwarz inequality, it follows from Propostion \ref{mainterms} and Proposition \ref{errordmv} that
\begin{equation}
\label{eq:main3}
\begin{split}
 \sum_{T<\gamma_{\!_f}\leq 2T} \big| L'(\rho_{\!_{f}},f) \big|^{2} &= \sum_{T<\gamma_{\!_f}\leq 2T} \Bigg| \sum_{n \le X} \frac{\alpha_f(n)}{n^{\rho_{_f}}}+  \psi_{\!_f}(\rho_{_f}) \sum_{n \le X} \frac{ \beta_{\bar{f},X}(n) }{n^{1-\rho_{_f}}} + \mathcal{E}\big(\rho_{\!_f}, f) \Bigg|^{2} 
 \\
 &= \sum_{T<\gamma_{\!_f}\leq 2T} \Bigg| \sum_{n \le X} \frac{\alpha_f(n)}{n^{\rho_{_f}}}+  \psi_{\!_f}(\rho_{_f}) \sum_{n \le X} \frac{ \beta_{\bar{f},X}(n) }{n^{1-\rho_{_f}}} \Bigg|^{2} 
 \\
 & \quad \quad \quad  + O\Big( \sqrt{\mathcal{A}_f(T)\mathcal{E}_f(T)} \Big) + O\Big( \sqrt{\mathcal{B}_f(T)\mathcal{E}_f(T)} \Big)+ O\Big( \mathcal{E}_f(T) \Big)
 \\
 &= \sum_{T<\gamma_{\!_f}\leq 2T} \Bigg| \sum_{n \le X} \frac{\alpha_f(n)}{n^{\rho_{_f}}}+  \psi_{\!_f}(\rho_{_f}) \sum_{n \le X} \frac{ \beta_{\bar{f},X}(n) }{n^{1-\rho_{_f}}} \Bigg|^{2} + O\!\left( \frac{T \log^4 T}{\sqrt{\log\log T}}\right).
\end{split}
\end{equation}
Here we have used the estimate $X\asymp T$. Thus, in order to prove Theorem \ref{2nd_moment}, it suffices to provide upper and lower bounds for the 
final sum in the above equation. 

We bound the sum on the right-hand side with the inequalities 
\begin{equation}
  \label{inequalities}
\left( \sqrt{ \sum_{n\le N} |u_n|^2 }- \sqrt{ \sum_{n\le N}  |v_n|^2 } \right)^{\!2} \ \le \ \sum_{n\le N} | u_n\!+\! v_n |^2 \ \le \ 
\left( \sqrt{ \sum_{n\le N} |u_n|^2 } + \sqrt{ \sum_{n\le N}  |v_n|^2 } \right)^{\!2},
\end{equation}
valid for complex numbers $\{u_n\}_{n=1}^{N}$ and $\{v_n\}_{n=1}^{N}$.  This is simply the triangle inequality in the form
\begin{equation}
  \label{triangleineq}
\Big| || {\bf x} ||_2 - || {\bf y} ||_2 \Big|  \le || {\bf x} \!+\! {\bf y} ||_2 \le || {\bf x} ||_2 + || {\bf y} ||_2
\end{equation}
where ${\bf x, y} \in \mathbb{C}^N$ and $|| \cdot ||_2$ denotes the $\ell^2$ norm. 
A key point in the argument is that if $|| {\bf x} ||_2 \ne || {\bf y} ||_2$, then \eqref{triangleineq} provides
a positive lower bound. 

Again appealing to the fact that $|\psi_{\!_f}(\rho_{\!_f}) |=1$ for all non-trivial zeros $\rho_{\!_f}$ of $L(s,f)$ and applying \eqref{inequalities} to the final sum on the right-hand side of \eqref{eq:main3}, it follows that
\begin{equation}
\begin{split} 
\bigg( \sqrt{\mathcal{A}_f(T)} &- \sqrt{\mathcal{B}_f(T)} \bigg)^2
\\
& \leq \sum_{T<\gamma_{\!_f}\leq 2T} \Bigg| \sum_{n \le X} \frac{\alpha_f(n)}{n^{\rho_{_f}}}+  \psi_{\!_f}(\rho_{_f}) \sum_{n \le X} \frac{ \beta_{\bar{f},X}(n) }{n^{1-\rho_{_f}}} \Bigg|^{2} 
\\
& \quad \quad \quad  \quad \leq \left( \sqrt{\mathcal{A}_f(T)} + \sqrt{\mathcal{B}_f(T)} \right)^2
\end{split}
\end{equation}
where $\mathcal{A}_f(T)$ and $\mathcal{B}_f(T)$ are given by \eqref{eq:AfBf}. 
From this observation, Proposition \ref{mainterms}, and \eqref{eq:main3}, we conclude that 
\begin{equation} \label{almostdone}
 (A_{\!_f}+o(1) ) \ \! T \log^4 X \leq \sum_{T<\gamma_{\!_f}\leq 2T} \big| L'(\rho_{\!_{f}},f) \big|^{2} \leq (B_{\!_f} +o(1)) \ \! T \log^4 X  
 \end{equation}
where 
$$  A_{\!_f} = \left(\sqrt{\frac{29}{24} } - \sqrt{\frac{5}{24} } \right)^2 \frac{c_f}{\pi} = \left(\frac{17-\sqrt{145}}{12\pi}\right)c_f  ,$$
$$ B_{\!_f} = \left(\sqrt{\frac{29}{24} } + \sqrt{\frac{5}{24} } \right)^2 \frac{c_f}{\pi} = \left(\frac{17+\sqrt{145}}{12\pi}\right)c_f,$$
and the $o(1)$ terms are $O\big(1/\sqrt{\log\log T} \ \! \big)$. Since $X\asymp T$, the theorem now follows by summing the estimates in \eqref{almostdone} over the dyadic intervals $(T/2,T], (T/4,T/2], (T/8, T/4], \ldots$.
\end{proof}

\smallskip

\noindent {\sc Remarks.} \\

\noindent{\it 1.} 
As the proof above indicates, it is the fact that the sums $\mathcal{A}_f(T)$ and $\mathcal{B}_f(T)$ in Proposition \ref{mainterms} have asymptotically different sizes which allow us to deduce this asymptotically positive lower bound for the average of $|L'(\rho_{_f},f)|^2$, the crucial estimate in to the proof of Theorem \ref{simple}.  It is tempting to try to use the method of Rudnick and Soundararajan \cite{RS2} to prove a lower bound of this order of magnitude.  However, this would require the asymptotic evaluation of the sum
$$\sum_{0 < \gamma_{_f} \le T} L'(\rho_{_f},f) \overline{A(\rho_{_f})}$$ where $A(s)$ is a short Dirichlet polynomial and, even in the case $A\equiv1$, this is a difficult open problem.  

\smallskip

\noindent{\it 2.} A main ingredient in the proof of Theorem \ref{2nd_moment} is the asymptotic evaluation of the sums in Proposition \ref{mainterms}.   One novelty of our proof is that we make no use of the Landau-Gonek type explicit formula for the sum
$\sum_{0 < \gamma_f \le T} x^{\rho_f}$ 
or of any version of the Guinand-Weil explicit formula while proving Proposition \ref{mainterms}.    Instead, we use standard Dirichlet polynomial techniques  to estimate these sums, namely the residue theorem and Montgomery and Vaughan's mean-value theorem (see Lemma \ref{MontgomeryVaughan} below). A similar argument was previously used by the authors in \cite{MN} and \cite{MN2}. As a consequence, we only require Deligne's bound  $|\lambda_f(n)|\leq d(n)$ and estimates for sums of the form $$\sum_{n \le x} |\lambda_f(n)|^2 \quad \text{ and } \quad \sum_{n \le x} |\lambda_f(n)|$$
which are classical, and we do not require estimates for the shifted convolution sums 
\begin{equation}
\label{scs}
\sum_{n \le x} \lambda_f(n) \overline{\lambda_f(n+r)} \quad \quad (\text{for }r \in \mathbb{N})
\end{equation}
which arise when using explicit formulae techniques to estimate the sums in Proposition \ref{mainterms}. The fact that we can avoid this off-diagonal analysis is somewhat surprising, but seems to be analogous to the work of Ramachandra \cite{Ram} and of Zhang \cite{Z1,Z2} where it is shown that one can derive asymptotic expressions for the fourth moment of the Riemann zeta-function and continuous second moment of degree two $L$-functions, respectively, on the critical line without the need for estimates for the shifted convolution sums in \eqref{scs}. The present situation is more involved than these previous cases because we are averaging over zeros (as opposed to a continuous average), so it is perhaps even more striking that we can appeal to the Montgomery and Vaughan's mean-value theorem for Dirichlet polynomials in lieu of explicit formula techniques combined with estimates for shifted convolution sums. 

\smallskip

\noindent {\it 4.}    Using a heuristic argument based on the `$L$-functions ratios conjectures' of Conrey, Farmer, and Zirnbauer \cite{CFZ} (see also Conrey and Snaith \cite{CS}), we arrive at the following conjecture. 

\smallskip

\noindent{\bf Conjecture.} {\it Let  $f\in H_k(q,\chi)$, let $c_f$ be the constant in \eqref{eq:cfdefinition}, and let $X = \frac{\sqrt{q}T}{2 \pi}$. Then, we have}
\begin{equation}
   \label{eq:2ndmomentconjecture}
    \sum_{0<\gamma_{\!_f}\leq T} \big| L'(\rho_{\!_{f}},f) \big|^{2} = \frac{2}{3\pi} c_f T \log^4 X + O\big( T \log^3 X \big) 
\end{equation}
{\it where the implied constant depends only on $f$.}

\smallskip

\noindent Note that this is consistent with Theorem \ref{2nd_moment} and is analogous to a result of Gonek \cite{Go1} which states that
$$\sum_{0<\Im(\rho) \leq T} \big| \zeta'(\rho)\big|^{2} = \frac{T}{24\pi} \log^{4}T + O(T \log^3T) $$ 
assuming the Riemann hypothesis where $\rho$ runs through the non-trivial zeros of the Riemann zeta-function. 
However, since $L(s,f)$ is a degree two $L$-function, establishing \eqref{eq:2ndmomentconjecture} is comparable to establishing the conjectural formula
$$ \sum_{0<\Im(\rho)\leq T} \big| \zeta'(\rho)\big|^{4}
= \frac{T}{2880 \pi^3} \log^9 T + O(T \log^8 T). 
$$
Such a result appears to be unattainable using current techniques without some significantly new ideas. (See \cite{Ng} for some history and some results in this direction.) Likewise, we expect that some substantially new ideas are necessary in order to establish the above conjecture for the second moment of $ L'(\rho_{\!_{f}},f)$.

\smallskip

\subsection{Conventions and Notation}
Throughout the this article, we use $s=\sigma+it$ to denote a complex variable. For $f\in H_k(q,\chi)$, we denote the generalized Riemann hypothesis for $L(s,f)$ by $\text{RH}_{f}$.  For functions $g(x)$ and $h(x)$ we interchangeably use the notations  $g(x)=O(h(x))$, $g(x) \ll h(x)$, or $h(x) \gg g(x)$  to mean that there exists a constant $M >0$ such that $|g(x)| \le M |h(x)|$ for all sufficiently large $x$. The constants implied in our big-$O$, $\ll$, and $\gg$ estimates are allowed to depend on $ f\in H_k(q,\chi)$ and thus on $k$ and $q$.  In particular, we  often use the estimates $\log X=\log T +O(1)$ and $\log X \asymp \log T$ where $  X = \frac{\sqrt{q}T}{2 \pi}$ is the length of the Dirichlet polynomials in the approximate functional equation for $L'(s,f)$ in \eqref{eq:afe0}. The letter $p$ is used to denote a prime number. Finally, for a set $S \subset \mathbb{R}$, we write $I_{S}(t)$ to denote the indicator function of $S$. That is, $I_S(t)=1$ if $t \in S$ and $I_S(t)=0$ otherwise.

\subsection{Organization of the article}

The remainder of this article is devoted to proving the approximate functional equation for $L'(s,f)$, Proposition \ref{mainterms}, Proposition \ref{errordmv}, Theorem \ref{2kth_moment}, and Theorem \ref{shifted}. 
In \textsection 2, we collect several important properties of modular forms and their $L$-functions.  In \textsection 3, we establish an approximate functional equation for $L'(s,f)$, state two results concerning the distribution of the zeros of $L(s,f)$, and prove inequalities for $L'(s,f)/L(s,f)$ and $\log|L(s,f)|$. In \textsection 4, we state and prove three mean-value estimates for Dirichlet polynomials averaged over the zeros of $L(s,f)$ which are used to establish Proposition \ref{mainterms}, Proposition \ref{errordmv}, and Theorem \ref{shifted}, respectively. In \textsection 5, we prove various asymptotic formulae and bounds for arithmetic sums  involving $\lambda_f(n)$ and related arithmetic functions. 
 In \textsection 6, we deduce Propositions \ref{mainterms} and \ref{errordmv}. 
  In \textsection 7, we   prove Theorem \ref{shifted}. Finally, in \textsection 8, we deduce Theorem \ref{2kth_moment} from Theorem \ref{shifted}.

\section{Properties of modular forms and $L$-functions}

We begin by recalling some basic definitions concerning modular forms.
The Hecke congruence subgroup of level $q$ for $\Gamma =  SL_2(\mathbb{Z})$ is defined to be 
$$ \Gamma_0(q) = \left\{  \ \left( \begin{array}{cc}
a & b \\
c & d  
 \end{array} \right) \in \Gamma : q|c \right\}. $$
For a Dirichlet character $\chi$ (mod $q$), the complex vector space $S_k\big(\Gamma_0(q),\chi\big)$  consists of  functions $f$ on the upper half plane which satisfy
\[
   f \Big( \frac{az+b}{cz+d} \Big) = \chi(d) (cz+d)^k f(z) \  \text{ for all }  \ \Big( \begin{array}{cc}
a & b \\
c & d  
 \end{array} \Big) \in \Gamma_0(q)
\]
and which are holomorphic at their cusps.   There is a distinguished basis, $ H_k(q,\chi)$ of $S_k\big(\Gamma_0(q),\chi\big)$, called the set of normalized holomorphic newforms $f \in S_k\big(\Gamma_0(q),\chi\big)$.  Each element $f \in H_k(q,\chi)$ has a Fourier expansion of the form
\begin{equation}
   \label{fourierexp}
   f(z)=\sum_{n=1}^\infty \lambda_f(n) n^{(k-1)/2} e^{2\pi i n z} 
\end{equation}
for $\Im(z) >0$ where $\lambda_f(1)=1$, and from this the $L$-function associated to $f$ is defined by 
\begin{equation}
  \label{eulerproduct}
L(s,f) = \sum_{n=1}^{\infty} \frac{\lambda_f(n)}{n^{s}} =  \!  \prod_{p \text{ prime}} \left(1-\frac{\lambda_f(p)}{p^s}+\frac{\chi(p)}{p^{2s}} \right)^{-1}
\end{equation}
for $\Re(s)>1$.  In the Fourier expansion   \eqref{fourierexp} of $f(z)$, the normalizing factor $n^{(k-1)/2}$ implies that the critical line for $L(s,f)$ is $\Re(s) =1/2$, rather than $\Re(s) =k/2$. The $L$-function in \eqref{eulerproduct} can be analytically continued to all of $\mathbb{C}$ and satisfies the functional equation, in asymmetric form,
\begin{equation}\label{eq:func_eqn}
 L(s,f) = \varepsilon_{\!_f} \psi_{\!_f}(s) L(1\!-\!s,\bar{f}) 
 \end{equation}
where the root number $\varepsilon_{_f}$ satisfies $|\varepsilon_{\!_f}|= 1$,
\begin{equation} \label{eq:psi}
  \psi_{\!_f}(s) =  \left(\frac{\sqrt{q}}{2\pi}\right)^{1-2s} \frac{\Gamma(1\!-\!s \! +\! \frac{k-1}{2})}{\Gamma(s\!+\!\frac{k-1}{2})},
\end{equation}
and  $\bar{f}(z)=\overline{f(\bar{z})}$, so that $L(s,\overline{f}) = \sum_{n=1}^{\infty} \lambda_{\bar f}(n)n^{-s}$
where $  \lambda_{\bar f}(n) = \overline{\lambda_f(n)}$. 
The Euler product representation of $L(s,f)$ in \eqref{eulerproduct} may be written as 
\[
L(s,f) = \prod_{p \text{ prime}} \left(1-\frac{r_{f}(p)}{p^{s}} \right)^{-1}\left(1-\frac{s_{f}(p)}{p^{s}} \right)^{-1}
\]
where $r_f(p)+s_f(p)=\lambda_f(p)$ and $r_f(p)s_f(p) =\chi(p).$ Moreover, by comparing the coefficient of $p^{-ms}$ term on the right-hand side of this identity to the Dirichlet series for $L(s,f)$, it follows that 
\begin{equation}
   \label{eq:lambdafpk}
   \lambda_f(p^m) = \sum_{\ell=0}^{m} r_f(p)^\ell s_f(p)^{m-\ell}.
\end{equation}
From the work of Deligne \cite[Th\'{e}or\`{e}me 8.2]{D}, it follows that $|r_{f}(p)|=|s_f(p)|=1$ if $p$ does not divide $q$.  Hence, for $(p,q)=1$, we have $    |\lambda_f(p^m)| \le m+1. $
On the other hand, if $p|q$ then $s_f(p) =0$ so that $\lambda_f(p)=r_f(p)$ and also 
$\lambda_f(p^m) = \lambda_f(p)^m$.   By  Li \cite[Theorem 3]{Li},  if $p$ divides  $q$, then
$|\lambda_f(p)| \le 1$ and thus
$ |\lambda_f(p^m)| \le 1.$ Since the Fourier coefficients $\lambda_{f}(n)$ are multiplicative, the above estimates combine to give the bound
\begin{equation}
   \label{eq:Delignesbound}
   |\lambda_f(n) | \le d(n)
\end{equation}
for every natural number $n$, where $d(n)$ denotes the number of positive divisors of  $n$.
\smallskip

Logarithmically differentiating the Euler product in \eqref{eulerproduct}, we have
$$ \frac{d}{ds} \log L(s,f):= \frac{L'}{L}(s,f) = - \sum_{n=1}^{\infty} \frac{\Lambda_{f}(n)}{n^{s}} $$
for $\Re(s)>1$, where the coefficients $\Lambda_{f}(n)$ are supported on prime powers.
A straightforward calculation shows that 
$$ \Lambda_{f}(p^{m}) = \big(r_f(p)^m+s_f(p)^m \big) \log p $$
for primes $p$ and natural numbers $m$. We remark that in the case $m=1$ we have $ \Lambda_{f}(p)=\lambda_{f}(p)\log p,$ and in the case $m=2$ we have $\Lambda_f(p^2)=(\lambda_f(p^2)-\chi(p))\log p$. In general, the inequality 
\begin{equation}  \label{eq:lambdafbound}
|\Lambda_{f}(n)|\leq 2 \Lambda(n)  
\end{equation}
holds for every positive integer $n$ where $\Lambda(n)$ denotes von Mangoldt's function.

\smallskip

\section{Some estimates for $L(s,f)$, its derivative, and its logarithmic derivative}

In this section, we collect a number of analytic estimates for $L(s,f)$, its derivative, and its logarithmic derivative. Recall, we define the quantity $N_{\!_f}(T)$ to be the number of non-trivial zeros $\rho_{\!_f}=\beta_{\!_f}+i\gamma_{\!_f}$ of $L(s,f)$ in the region $0\leq\beta_{\!_f}\leq1$ and $0\leq \gamma_{\!_f} \leq T$, counted with multiplicity. If $t$ does not correspond to any ordinate $\gamma_f$ of a zero $\rho_f$ of $L(s,f)$, let 
$$S_{\!_f}(t) =  -\frac{1}{\pi} \Im 
\Big(
\int_{1/2}^\infty \frac{L'}{L}(\sigma+it,f) \ \! d\sigma
\Big)
$$
and otherwise let
$$ S_{\!_f}(t) = \frac{1}{2}\lim_{\varepsilon \to 0} \big( S_{\!_f}(t+\varepsilon) + S_{\!_f}(t-\varepsilon)   \big).$$
Then the following formula for $N_{\!_f}(T)$ holds.

\begin{lemma} \label{NfTSfT}
For $t\geq 10$, we have $N_{\!_f}(t)=\vartheta_{\!_f}(t)+S_{\!_f}(t)$ where $\vartheta(t)$ is continuously differentiable, 
\[
   \vartheta_{\!_f}(t) =  \frac{t}{\pi} \log
   \Big( \frac{\sqrt{q}t}{2 \pi e} \Big) + O(1), 
\]
$\vartheta'_{\!_f}(t) = O(\log t)$, and $S_{\!_f}(t)=O(\log t)$ where the implied constants depend only of $f$. Moreover, on $\text{RH}_f$, we have the stronger bound
\begin{equation} \label{eq:SfTbound}
S_{\!_f}(t) = O\!\left(\frac{\log t}{\log\log t}\right).
\end{equation}
\end{lemma}
\begin{proof}
The unconditional estimate for $N_{\!_f}(t)$ can be found in Theorem 5.38 of Iwaniec and Kowalski \cite{IK}. 
The estimate for $ S_{\!_f}(t)$ assuming $\text{RH}_f$ follows from work of Selberg. In particular, equation ($2.2^\prime$) of \cite{Se4} implies that, for $2\leq x\leq t^2$, 
\begin{equation*}
S_{_f}(t) = -\frac{1}{\pi} \sum_{n<x^2} \frac{\Lambda_f(n) \theta_x(n) \sin(t \log n)}{n^{\frac{1}{2}} \log n} + O\left(\frac{1}{\log x} \Bigg|\sum_{n<x^2} \frac{\Lambda_f(n) \theta_x(n)}{n^{\sigma_x+it}} \Bigg| \right) +O\left(\frac{\log |t|}{\log x}\right)
\end{equation*}
 where $\sigma_x=\frac{1}{2}+\frac{1}{\log x}$, $\theta_x(n) = 1$ if $1\leq n \leq x$, $\theta_x(n) = 2{-}\frac{\log n}{\log x}$ if $x\leq n \leq x^2$, and $\theta_x(n)=0$ otherwise. Choosing $x = \sqrt{\log t}$, using the inequality $|\Lambda_f(n)| \leq 2 \Lambda(n)$, and estimating trivially, we deduce \eqref{eq:SfTbound}.
\end{proof}

\smallskip

\begin{lemma}  \label{LprimeLapprox}
For each $T\geq 10$ there exists a real number $\tau$ satisfying $T\leq \tau \leq T+1$ such that
\begin{equation}\label{log_deriv_seq}
 \left|\frac{L'}{L}(\sigma+i \tau,f) \right| = O\big(\log^2 T\big)
 \end{equation}
uniformly for $-1\leq \sigma \leq 2$ and, assuming $\text{RH}_f$, that
\begin{equation}\label{log_deriv_int}
 \int_{-1}^2 \left|\frac{L'}{L}(\sigma+i \tau,f) \right| \ \! d\sigma = O\big( \log T \big).
 \end{equation}
\end{lemma}

\smallskip

\noindent{\sc Remark.} Unconditionally, we can show that the integral in \eqref{log_deriv_int} is  $O(\log T \log\log T)$.

\smallskip

\begin{proof} 
Proposition 5.7 of Iwaniec and Kowalski \cite{IK} implies that if $T \ge 10$ and  $T\leq t \leq 2T$, then
\begin{equation}\label{logderiv} 
\frac{L'}{L}(s,f) \ = \sum_{|\gamma_{\!_f}-t|\leq 1} \frac{1}{s\!-\!\rho_{\!_f}} + O(\log T)
\end{equation}
uniformly for $-1\leq \sigma \leq 2$. Now Lemma \ref{NfTSfT} implies that $N_{\!_f}(T+1)- N_{\!_f}(T) \ll \log T$ uniformly in $T$, and thus there exists $\tau \in [T,T+1]$ such that  $|\gamma_{_f} - \tau| \gg (\log T)^{-1}$
for all non-trivial zeros $\rho_{_f}$ of $L(s,f)$.   For each term in the sum we have
\begin{equation*}
 \frac{1}{|s-\rho_{_f}|} \ll \frac{1}{|\gamma_{_f}-\tau|} \ll \log T,
 \end{equation*} 
and since the number of zeros being summed is $O(\log T)$ the first estimate \eqref{log_deriv_seq} in the lemma now follows. 

\smallskip

Now assume $\text{RH}_f$. Then, by Lemma \ref{NfTSfT}, we know that $N_{\!_f}(T+1/\log\log T)- N_{\!_f}(T) \ll \log T/\log\log T.$ From this, by modifying the proof of Proposition 5.7 of \cite{IK} in a straightforward manner, for $T \ge 10$ and  $T\leq t \leq 2T$, we derive that 
\begin{equation}\label{logderiv2} \frac{L'}{L}(s,f) \ = \sum_{|\gamma_{\!_f}-t|\leq \frac{1}{\log\log T}} \frac{1}{s\!-\!\rho_{\!_f}} + O(\log T)
\end{equation}
uniformly for $-1\leq \sigma \leq 2$, and the number of terms in the sum is $O(\log T/\log\log T)$. Now let $I$ denote the integral in \eqref{log_deriv_int} that we wish to estimate. Then $I=I_1 + I_2$ where
$$ I_1= \int_{-1}^{\frac{1}{2}} \left| \frac{L'}{L}(\sigma\!+\!i\tau,f) \right| \ \! d\sigma \quad \text{ and } \quad I_2=\int_{\frac{1}{2}}^2 \left| \frac{L'}{L}(\sigma\!+\!i\tau,f) \right| \ \! d\sigma. $$ 
By the functional equation for $L(s,f)$ and Stirling's formula for $\Gamma'(s)/\Gamma(s)$, we have
$$ \frac{L'}{L}(s,f)=-\frac{L'}{L}(1\!-\!s,\bar{f}) + O(\log t)$$
uniformly for $-1\leq \sigma \leq 2$. After a change of variables, it follows that $I_1= I_2 + O(\log T)$ and hence that $I=2 I_2 + O(\log T).$ By \eqref{log_deriv_seq}, estimating trivially, we have
\begin{eqnarray*}
\int_{\frac{1}{2}}^{\frac{1}{2}+\frac{1}{\log T}} \left| \frac{L'}{L}(\sigma\!+\!i\tau,f) \right| \ \! d\sigma \ \ll \ \log T.
\end{eqnarray*}
Therefore, by \eqref{logderiv2}, we have 
\begin{equation*}
\begin{split}
 I &=   2 \int_{\frac{1}{2}+\frac{1}{\log T}}^2 \left| \frac{L'}{L}(\sigma\!+\!i\tau,f) \right| \ \! d\sigma + O\big(\log T\big) 
 \\
  &\ll   \sum_{|\gamma_{_f}-\tau|\leq \frac{1}{\log\log \tau}} \left[  \int_{\frac{1}{2}+\frac{1}{\log T}}^2 \frac{d\sigma}{\sigma\!-\!\frac{1}{2}} \right] \ \! + \ \! \log T
 \\
 &\ll \!\!\! \sum_{|\gamma_{_f}-\tau|\leq \frac{1}{\log\log \tau}} \log\log T \ \! + \ \! \log T
 \\
 &\ll  \ \log T
\end{split}
 \end{equation*}
since there are $O(\log T/\log\log T)$ terms in the sum. This completes the proof of the lemma. 
\end{proof}

\smallskip

In order to prove Theorem \ref{shifted} (and thus Theorem \ref{2kth_moment}), we require the following version of the Landau-Gonek explicit formula for the zeros of $L(s,f)$. This formula is not used in the proof of Theorem \ref{2nd_moment}.

\begin{lemma}\label{gonek}
Let $x,T>1$ and let $\rho_{\!_f}=\beta_{_f}+i\gamma_{_f}$ denote a non-trivial zero of $L(s,f)$.  Then
\begin{equation*}
\begin{split}
\sum_{0<\gamma_f \leq T} x^{\rho_{\!_f}} &= -\frac{T}{2\pi}\Lambda_f(x) + O\big(x\log(2xT)\log\log(3x)\big)
\\
&\quad + O\left(\log x \ \! \min\left(T,\frac{x}{\langle x\rangle}\right)\right) + O\left(\log(2T)\min\Big(T,\frac{1}{\log x}\Big)\right), 
\end{split}
\end{equation*}
where $\langle x\rangle$ denotes the distance from $x$ to the nearest prime power other than $x$ itself, $\Lambda_f(x)$ is the arithmetic function defined in section 2 if $x$ is a positive integer, and $\Lambda_f(x)=0$ otherwise.
\end{lemma}
\begin{proof}
This can be proved {\it mutatis mutandis} following the proof of Theorem 1 of Gonek \cite{Go2} using the estimate for $L'(s,f)/L(s,f)$ in  \eqref{logderiv} and the inequality $|\Lambda_f(n)|\leq 2 \Lambda(n)$.
\end{proof}

\begin{lemma} \label{approx}
Let $s=\frac{1}{2}+it$ and let $T\geq 10$. Then, for 
$X = \frac{\sqrt{q} T}{2 \pi}$ and  $T\leq t \leq 2T$, we have
\begin{equation}
\begin{split}
L'(s,f) = \sum_{n\leq X} \frac{\alpha_f(n)}{n^s} + \varepsilon_{_f}  \psi_{\!_f}(s) \sum_{n\leq X} \frac{\beta_{\bar{f},X}(n)}{n^{1-s}} + \mathcal{E}(s,f)
\end{split}
\end{equation}
where $\alpha_f(n)=-\lambda_f(n) \log n$, $\beta_{\bar{f},X}(n)=-\lambda_{\bar{f}}(n)\log(X^2/n)$, and $\displaystyle{ \mathcal{E}(s,f) =  \sum_{j=1}^5 \mathcal{E}_j(s,f)}$ where 
\begin{equation}
\begin{split}
 \label{eq:Esf}
 \mathcal{E}_1(s,f) & =  \sum_{n\leq X} \frac{\alpha_f(n)}{n^s}\big(e^{-\frac{n}{X}}\!-\!1\big), \
  \mathcal{E}_2(s,f)  = \sum_{n> X} \frac{\alpha_f(n)}{n^s} e^{-\frac{n}{X}}, \ \mathcal{E}_3(s,f)  = O\Big(  \Big| \sum_{n\leq X} \frac{\lambda_f(n)}{n^s} \Big|\Big),
 \\
  \mathcal{E}_4(s,f) & =  - \frac{ \varepsilon_{_f}}{2\pi i} \int_{-\frac{3}{4}-i\infty}^{-\frac{3}{4}+i\infty} \Gamma(w) X^w \frac{d}{ds} \Bigg[ \psi_{\!_f}(s\!+\!w)\sum_{n>X} \frac{\lambda_{\bar{f}}(n)}{n^{1-s-w}}\Bigg] \ \! dw, \ \text{ and} 
\\
  \mathcal{E}_5(s,f) & = - \frac{ \varepsilon_{_f}}{2\pi i} \int_{\frac{1}{4}-i\infty}^{\frac{1}{4}+i\infty} \Gamma(w) X^w \frac{d}{ds} \Bigg[   \psi_{\!_f}(s\!+\!w)\sum_{n\leq X} \frac{\lambda_{\bar{f}}(n)}{n^{1-s-w}}\Bigg] \ \! dw.
\end{split}
\end{equation}
\end{lemma}

\begin{proof}  
The proof of this lemma is based upon Ramachandra's proof of the approximate functional equation for the square of the Riemann zeta-function and related functions (see Theorem 2 of  \cite{Ram}). We begin with the Mellin transform identity 
\[
  e^{-t} = \frac{1}{2 \pi i} \int_{2-i\infty}^{2+i\infty} \Gamma(w) \, t^{-w} \ \! dw 
\]
for $t >0$, which implies that 
\[
   \ \sum_{n=1}^\infty \frac{\lambda_f(n)}{n^s}e^{-\frac{n}{X}} \ = \
 \frac{1}{2\pi i} \int_{2-i\infty}^{2+i\infty} L(s\!+\!w,f) \Gamma(w) X^w \ \! dw.
 \]
This formula follows by expanding $L(s+w,f)$ as an absolutely convergent Dirichlet series, interchanging the sum and the integral, and then integrating term-by-term.  Next, we shift the line of integration of the integral left from $\Re(w)=2$ to $\Re(w)=-\frac{3}{4}$ and derive that 
\[
   \sum_{n=1}^\infty \frac{\lambda_f(n)}{n^s}e^{-\frac{n}{X}} \ = \ L(s,f) + 
 \frac{1}{2\pi i} \int_{-\frac{3}{4}-i\infty}^{-\frac{3}{4}+i\infty} L(s\!+\!w,f) \Gamma(w) X^w \ \! dw \, ;
\]
the first term on the right-hand side comes from the the simple pole of the integrand at $w=0$ which has a residue of $L(s,f)$.  We decompose the new integral along $\Re(w)=-\frac{3}{4}$ into two integrals by using the functional equation $
L(s\!+\!w,f)= \varepsilon_{_f} \psi_{f}(s\!+\!w)L(1\!-\!s\!-\!w, \bar{f})$ and then 
by absolute convergence we write
\[
     L(1-s-w,\bar{f}) = 
  \sum_{n \le X} \frac{\lambda_{\bar{f}}(n)}{n^{1-s-w}} + \sum_{n > X} \frac{\lambda_{\bar{f}}(n)}{n^{1-s-w}}. 
\]
We now shift the line of integration in the first integral containing the sum with the terms $n\leq X$ right from $\Re(w)=-\frac{3}{4}$ to $\Re(w)=\frac{1}{4}$. Again we pass over the pole of the $\Gamma(w)$ at $w=0$. Since the residue of the integrand at $w=0$ in this integral is 
\[
 - \varepsilon_{_f} \psi_{\!_f}(s)\sum_{n\leq X}\frac{\lambda_{\bar{f}}(n)}{n^{1-s}},
\] 
by collecting and rearranging terms we establish 
 \begin{equation}\label{L}
\begin{split}
L(s,f) \ &= \ \sum_{n=1}^\infty \frac{\lambda_f(n)}{n^s}e^{-\frac{n}{X}} +  \varepsilon_{_f} \psi_{\!_f}(s)\sum_{n\leq X}\frac{\lambda_{\bar{f}}(n)}{n^{1-s}}
\\
& \quad \quad- \frac{1}{2\pi i} \int_{-\frac{3}{4}-i\infty}^{-\frac{3}{4}+i\infty}   \varepsilon_{_f}\psi_{\!_f}(s\!+\!w)\sum_{n>X} \frac{\lambda_{\bar{f}}(n)}{n^{1-s-w}}  \Gamma(w) X^w\ \!  dw
\\
& \quad \quad - \frac{1}{2\pi i} \int_{\frac{1}{4}-i\infty}^{\frac{1}{4}+i\infty}  \varepsilon_{_f}\psi_{\!_f}(s\!+\!w)\sum_{n\leq X} \frac{\lambda_{\bar{f}}(n)}{n^{1-s-w}} \Gamma(w) X^w \ \!  dw.
\end{split}\end{equation}
Differentiating both sides of this expression, we deduce that
\begin{equation}\label{L_prime}
\begin{split}
L'(s,f) &= \sum_{n=1}^\infty \frac{\alpha_f(n)}{n^s}e^{-\frac{n}{X}} + \varepsilon_{_f} \psi_{\!_f}(s) \left\{\frac{\psi_{\!_f}'}{\psi_{\!_f}}(s)\sum_{n\leq X}\frac{\lambda_{\bar{f}}(n)}{n^{1-s}}+\sum_{n\leq X} \frac{\lambda_{\bar{f}}(n) \log n }{n^{1-s}} \right\}
\\
& \quad \quad - \frac{1}{2\pi i} \int_{-\frac{3}{4}-i\infty}^{-\frac{3}{4}+i\infty} \Gamma(w) X^w \frac{d}{ds} \Bigg[ \varepsilon_{_f}  \psi_{\!_f}(s\!+\!w)\sum_{n>X} \frac{\lambda_{\bar{f}}(n)}{n^{1-s-w}}\Bigg] \ \! dw
\\
& \quad\quad - \frac{1}{2\pi i} \int_{\frac{1}{4}-i\infty}^{\frac{1}{4}+i\infty} \Gamma(w) X^w \frac{d}{ds} \Bigg[ \varepsilon_{_f}  \psi_{\!_f}(s\!+\!w)\sum_{n\leq X} \frac{\lambda_{\bar{f}}(n)}{n^{1-s-w}}\Bigg] \ \! dw
\end{split}
\end{equation}
where $\alpha_f(n)=-\lambda_f(n)\log n$.  Logarithmically differentiating \eqref{eq:psi}, we find that 
\[
     \frac{\psi_{\!_f}'}{\psi_{\!_f}}(s) = -2 \log ( \tfrac{\sqrt{q}}{2 \pi} ) - \frac{\Gamma'}{\Gamma}(1-s+ \tfrac{k-1}{2})
         - \frac{\Gamma'}{\Gamma}(s+ \tfrac{k-1}{2}). 
\]
Since $\frac{\Gamma'}{\Gamma}(z) = \log|z| + O(1)$ as long as  $|z| \ge \varepsilon$ and $|\arg(z)| \le \pi - \varepsilon$ for any fixed $\varepsilon>0$, it follows that
\[
\label{eq:psiprimepsistirling}
 \frac{\psi_{\!_f}'}{\psi_{\!_f}}\left(\sigma+\!it\right) = - \log X^2 + O(1)
\]
uniformly for $-1 \le \sigma \le 1$, $T\leq t \leq 2T$, and $T\geq 10$.
This implies that
\begin{equation}
   \label{eq:2ndtermidentity}
   \frac{\psi_{\!_f}'}{\psi_{\!_f}}(s)\sum_{n\leq X}\frac{\lambda_{\bar{f}}(n)}{n^{1-s}}+\sum_{n\leq X} \frac{\lambda_{\bar{f}}(n) \log n }{n^{1-s}}
   = - \sum_{n \leq X}  \frac{\lambda_{\bar f}(n)\log(X^2/n)}{n^{1-s}}
   + O\Bigg( \bigg| \sum_{n \leq X}  \frac{\lambda_{\bar f}(n)}{n^{1-s}} \bigg| \Bigg)
\end{equation}
for $s=\frac{1}{2}+it$, $T\leq t \leq 2T$, and $T\geq 10$.
Now, observing that $|\varepsilon_{_f} \psi_{\!_f}(\frac{1}{2}+it)|=1$, the lemma follows from \eqref{L_prime} and \eqref{eq:2ndtermidentity} since
\[
 \Bigg|\sum_{n\leq X}\frac{\lambda_{\bar{f}}(n)}{n^{1-s}} \Bigg| = \Bigg|\sum_{n\leq X}\frac{\lambda_f(n)}{n^{s}} \Bigg|
 \]
for $s=\frac{1}{2}+it$.
\end{proof}

\smallskip

The final lemma in this section is  an inequality for $\log|L(\sigma+it,f)|$ which holds uniformly for $\sigma$ in the interval $[\frac{1}{2},\frac{3}{4}]$ and does not involve the zeros of $L(s,f)$. This inequality is essential in the proof of Theorem \ref{shifted}. To state this result, we let $\mu_0=0.5671\ldots$ be the unique positive real number satisfying $e^{-\mu_0} = \mu_0$ and put $\sigma_\mu=\sigma_{\mu,x} = 1/2+\mu/\log x$. Then the following inequality holds.

\begin{lemma}  \label{1st_inequality}
Let $f\in H_k(q,\chi)$ and assume $\text{RH}_f$. Suppose $T$ is large, let $t \in [T,2T]$, and let $3\le x\le T^2$. Then, for all $\mu_0\le \mu \le \frac{1}{4}\log x$, the estimate
\[
\log|L(\sigma\!+\!it,f)| \le \Re \sum_{n\le x} \frac{\Lambda_f(n)}{n^{\sigma_\mu+it}\log n} \frac{\log(x/n)}{\log x} + \big(1+\mu\big) \frac{\log T}{\log x} + O(1)
\]
holds uniformly for $\frac{1}{2}\le \sigma \leq \sigma_\mu \le \frac{3}{4}.$
\end{lemma}

\begin{proof} 
This lemma, which is inspired by the main proposition in Soundararajan's paper \cite{S}, can be established in a similar manner to Lemma 2.1 of \cite{M}. For $s$ not coinciding with a zero of $L(s,f)$, let
\[
F(s) = \Re \sum_{\rho_{\!_f}} \frac{1}{s\!-\!\rho_{\!_f}} =\sum_{\gamma_{\!_f}} \frac{\sigma\!-\!1/2}{(\sigma\!-\!1/2)^2\!+\!(t\!-\!\gamma_{\!_f})^2}.  
\]
Since $F(s) \ge 0$ if $\sigma \ge 1/2$, it follows from Hadamard's factorization formula for $L(s,f)$ (cf. Theorem 5.6 and equation (5.86) of \cite{IK}) and Stirling's formula for $\Gamma'(s)/\Gamma(s)$ that
\begin{equation}\label{log1}
\begin{split}
 - \Re \frac{L'}{L}(s,f) &= \frac{1}{2} \log \frac{q}{\pi^2} + \frac{1}{2} \Re \frac{\Gamma'}{\Gamma}\Big(\tfrac{s}{2}\!+\!\tfrac{(k-1)}{4} \Big) + \frac{1}{2} \Re \frac{\Gamma'}{\Gamma}\Big(\tfrac{s}{2}\!+\!\tfrac{(k+1)}{4} \Big) - F(s)
 \\
 &\leq \log T - F(s) + O(1)
 \\
 &\leq \log T + O(1)
\end{split}
\end{equation}
uniformly for $T\le t \leq 2T$, $\frac{1}{2} \leq \sigma \le \frac{3}{4}$, and $T$ sufficiently large. Consequently, the inequality 
\begin{equation} \label{log2}
\begin{split}
\log|L(\sigma\!+\!it,f)| - \log|L(\sigma_\mu\!+\!it,f)| &= \int_\sigma^{\sigma_\mu} \left[- \Re \frac{L'}{L}(u\!+\!it,f)\right] \ \!du
\\
&\leq (\sigma_\mu-\sigma) \big(\log T+ O(1)\big)
\\
&\leq \big(\sigma_\mu-\tfrac{1}{2}\big) \big(\log T+ O(1)\big)
\\
&\leq \frac{\mu \log T}{\log x}+ O( 1 )
\end{split}
\end{equation}
holds uniformly for $\frac{1}{2}\le \sigma \leq \sigma_\mu \le \frac{3}{4}$ and $T\le t \le 2T$. Note that the assumption $\mu \le \frac{1}{4} \log x$ implies that $\sigma_\mu \le \frac{3}{4}$. In order to complete the proof of the lemma, we require an upper bound for  $\log|L(\sigma_\mu\!+\!it,f)|$. We deduce such a bound from the following identity of Chandee (Lemma 2.4 of \cite{Chandee}) which states that
\begin{equation*} 
\begin{split}
-\frac{L'}{L}(s,f)  = \sum_{n\le x} & \frac{\Lambda_f(n)}{n^s} \frac{\log(x/n)}{\log x} + \frac{1}{\log x} \left(\frac{L'}{L}\right)'\!(s,f) 
\\
&+ \frac{1}{\log x} \sum_{\rho_{\!_f}} \frac{x^{\rho_{\!_f}-s}}{(\rho_{\!_f}\!-\!s)^2} + \frac{1}{\log x} \sum_{n=1}^\infty \frac{x^{-n-(k-1)/2-s}}{(n\!-\!(k\!-\!1)/2\!-\!s)^2}.
\end{split}
\end{equation*}
for $s$ not coinciding with a zero of $L(s,f)$. This is an analogue of a corresponding identity for the logarithmic derivative of the Riemann zeta-function which was proved by Soundararajan in \cite{S}. Integrating both sides of this expression over $\sigma$ from $\sigma_\mu$ to $\infty$ and using the assumption that $3 \le x \le T^2$, it follows that 
\begin{equation} \label{log3}
\begin{split}
 \log|L(\sigma_\mu\!+\!it,f)| = \Re \sum_{n\le x} &\frac{\Lambda_f(n)}{n^{\sigma_\mu + it} \log n} \frac{\log(x/n)}{\log x} + \frac{1}{\log x} \Re  \frac{L'}{L}(\sigma_\mu\!+\!it,f) 
 \\
 &+ \frac{1}{\log x} \sum_{\rho_{\!_f}} \Re \int_{\sigma_\mu}^\infty \frac{x^{\rho_{\!_f}-s}}{(\rho_{\!_f}\!-\!s)^2}\ \! d\sigma +  O\Big(\frac{1}{\log x}\Big).
\end{split}
\end{equation}
We now estimate the second and third terms on the right-hand side of this expression. By the second line in \eqref{log1}, it follows that
\[
 - \Re \frac{L'}{L}(\sigma_\mu\!+\!it,f)  \leq \log T - F(\sigma_\mu\!+\!it) + O(1)
\]
for $T\le t \le 2T$ and $\mu \le \frac{1}{4} \log x$. Using this inequality, and observing that
\begin{equation*}
\begin{split}
\sum_{\rho_{\!_f}} \Big| \int_{\sigma_\mu}^\infty \frac{x^{\rho_{\!_f}-s}}{(\rho_{\!_f}\!-\!s)^2}\ \! d\sigma \Big| \leq \sum_{\rho_{\!_f}} \int_{\sigma_\mu}^\infty \frac{x^{1/2-\sigma}}{|\rho_{\!_f}\!-\!s|^2}\ \! d\sigma \leq  \sum_{\rho_{\!_f}} \frac{x^{1/2-\sigma_\mu}}{|\rho_{\!_f}\!-\!\sigma_\mu\!-\!it|^2\log x} = \frac{x^{1/2-\sigma_\mu}F(\sigma_\mu\!+\!it) }{(\sigma_\mu\!-\!1/2) \log x } 
\end{split}
\end{equation*}
we deduce from \eqref{log3} that
\begin{equation*} 
\begin{split}
\log|L(\sigma_\mu\!+\!it,f)| \le \Re \sum_{n\le x} & \frac{\Lambda_f(n)}{n^{\sigma_\mu + it} \log n} \frac{\log(x/n)}{\log x} + \frac{\log T}{\log x} 
\\
&+ \frac{F(\sigma_\mu\!+\!it)}{\log x}\left( \frac{x^{1/2-\sigma_\mu}}{(\sigma_\mu\!-\!1/2) \log x } - 1 \right) + O\Big(\frac{1}{\log x}\Big).
\end{split}
\end{equation*}
Adding this inequality to \eqref{log2}, we obtain that
\begin{equation*}
\begin{split}
\log|L(\sigma\!+\!it,f)| \le \Re \sum_{n\le x} &\frac{\Lambda_f(n)}{n^{\sigma_\mu + it} \log n} \frac{\log(x/n)}{\log x} + \big(1+\mu\big)\frac{\log T}{\log x} 
\\
&+ \frac{F(\sigma_\mu\!+\!it)}{\log x}\left( \frac{x^{1/2-\sigma_\mu}}{(\sigma_\mu\!-\!1/2) \log x } - 1 \right) + O(1).
\end{split}
\end{equation*}
If $\mu \ge \mu_0$, then the third term on the right-hand side involving $F(\sigma_\mu\!+\!it)$ is less than or equal to zero, and hence omitting it does not change the inequality. The lemma now follows.
\end{proof}

\noindent{\sc Remark.} Since the coefficients $\Lambda_f(n)$ are supported on prime powers and satisfy  $|\Lambda_f(n)| \le 2 \Lambda(n)$, it follows from the previous lemma that 
\begin{equation}\label{prime_inequality}
\begin{split}
\log|L(\sigma\!+\!it,f)| \ \le \ \Re \sum_{p \le x} &\frac{\lambda_f(p)}{p^{\sigma_\mu+it}} \frac{\log(x/p)}{\log x} + \Re \sum_{p \le \sqrt{x} } \frac{(\lambda_f(p^2)\!-\!\chi(p))}{p^{2\sigma_\mu+2it}} \frac{\log(\sqrt{x}/p)}{\log\sqrt{x}} 
\\
&\quad + \big(1+\mu\big) \frac{\log T}{\log x} + O(1)
\end{split}
\end{equation}
holds uniformly for sufficiently large $T$ when $T \le t \le 2T$, $3\le x\le T^2$, $\frac{1}{2}\le \sigma \leq \sigma_\mu \le \frac{3}{4}$,
and $\mu \ge \mu_0$. Here we have used the identities $\Lambda_f(p) = \lambda_f(p) \log p$ and $\Lambda_f(p^2) = (\lambda_f(p^2)\!-\!\chi(p)) \log p$ which, as mentioned in \textsection 2, hold for all primes $p$. Note that
\[
\big| \lambda_f(p^2)\!-\!\chi(p) \big|^2 = \left|\frac{\Lambda_f(p^2)}{\log p}\right|^2 \leq \left|\frac{2 \Lambda(p^2)}{\log p}\right|^2 = \left|\frac{2 \Lambda(p)}{\log p}\right|^2=4.
\]
This bound is applied in \textsection 7 during the proof of a lemma which is used to establish Theorem \ref{shifted}.

\section{Some mean-value estimates}

In this section, we state and prove three propositions which are used to establish Proposition \ref{mainterms}, Proposition \ref{errordmv}, and Theorem \ref{shifted}, respectively. In order to state the first proposition, we introduce some notation. 
Let $T >0$ and let $A(s)$ be a Dirichlet polynomial defined by 
\begin{equation}
 \label{eq:dirichletpolynomial}
  A(s) = \sum_{n \le Y} \frac{a(n)}{n^s}
\end{equation}
where the coefficients $a(n)$ are complex numbers and let $\overline{A}(s) = \overline{A(\bar{s})}$
where $Y \asymp T$.  We assume the coefficients $a(n)$ satisfy the following two conditions:  there exists a $\eta$ satisfying $0<\eta \leq \frac{1}{2}$ such that
\begin{equation}
  \label{eq:coefficientscondition1}
   \sum_{n \le x} |a(n)|   \ll x(\log xT) (\log x)^{-\eta}
\end{equation}
and that
\begin{equation}
 \label{eq:coefficientscondition2}
   \sum_{n \le x} |a(n)|^2   \ll x (\log xT)^2
\end{equation}
uniformly for  $ x \ge 1$. These assumptions are made so that we can simultaneously handle sums involving $\alpha_f(n)$ and $\beta_{f,X}(n)$.

\smallskip

\begin{proposition} \label{dmva}
Assume $\text{RH}_{f}$. Let $T >0$, $X=\frac{\sqrt{q} T}{2 \pi}$, $Y \asymp T$,  and let $A(s)$ be a Dirichlet polynomial as defined in \eqref{eq:dirichletpolynomial} with coefficients $a(n)$
satisfying \eqref{eq:coefficientscondition1} and  \eqref{eq:coefficientscondition2}. 
 Then 
\begin{equation}
\begin{split}
  \label{eq:discretemeanformula}
  \sum_{T < \gamma_{\!_f} \leq 2T} |A(\rho_{\!_f})|^2  & = \frac{T}{\pi} 
  \left( 
  \log X   \sum_{n \le Y } \frac{|a(n)|^2}{n}  -   \Re\Bigg(  \sum_{n \le Y}  \frac{(\Lambda_f* a)(n) \overline{a(n)}}{n}  \Bigg)
  \right) \\
  & \quad \quad+ O\!\left( 
   T (\log T)^{4-2\eta} + T\log T \sqrt{ \sum_{n=1}^{\infty} \frac{|(\Lambda_f*a)(n)|^2}{n^{1+\frac{1}{\log T}}} } \ \!   \right).
\end{split}  
\end{equation}
\end{proposition}

\smallskip

This next simple proposition is used to prove Proposition \ref{errordmv}, showing that the mean-square of the error term $\mathcal{E}_f(T)$ in equation \eqref{eq:Ef} is smaller on average than the sums $\mathcal{A}_f(T)$ and $\mathcal{B}_f(T)$
in \eqref{eq:AfBf} by a factor of $\log \log T$. Although it is possible to use Proposition \ref{dmva} to prove Proposition \ref{errordmv}, it is considerably simpler to instead apply Proposition \ref{dmvbd}. 

\begin{proposition}  \label{dmvbd}
Assume $\text{RH}_{f}$. Let $B(t) = \sum_{n=1}^{\infty} b(n) n^{-it}$ be a Dirichlet polynomial with complex coefficients $b(n)$.  Then 
\begin{equation}
 \label{eq:dmvbd2} 
  \sum_{T < \gamma_{\!_f} \le 2T} |B(\gamma_{\!_f})|^2  \ll
    \frac{\log T}{\log  \log T} \Big( \sqrt{S_1 S_2} + S_{3}^2 \Big) +S_1   \log T 
\end{equation}
where 
\begin{align*}
  S_1 = \sum_{n=1}^{\infty} |b(n)|^2   \ \! (T + n ) \ \!  , \quad 
  S_2 = \sum_{n=1}^{\infty} |b(n) |^2  (T + n) \log^2 n, \quad \text{and} \quad 
  S_3 = \sum_{ n=1}^{\infty} |b(n)|. 
\end{align*}
\end{proposition}

\smallskip

The third proposition of this section is a pair of mean-value estimates for high powers of powers of Dirichlet polynomials supported on the primes.
These estimates are used in the proof of Theorem \ref{shifted}.

\begin{proposition} \label{dmv:primes}
Let $f\in H_k(q,\chi)$ and assume $\text{RH}_f$. Let $T$ be large, let $2\leq x \leq T$, let $w$ be a complex number with $\Re(w) \ge 0$, and let $m$ be a natural number such that $x^m \le T^{2/3}$. Then, for any complex numbers $a(p)$, we have
\begin{equation} \label{1st_prime_sum}
  \sum_{0<\gamma_{\!_f}\leq T} \Bigg| \sum_{p\le x} \frac{a(p)}{p^{  \rho_{{\!_f}}+w}} \Bigg|^{2m} \ll m ! \ \! N_{\!_f}(T) \Bigg( \sum_{p \le x} \frac{|a(p)|^2}{p} \Bigg)^{\! m}.
  \end{equation}
and 
\begin{equation} \label{2nd_prime_sum}
  \sum_{0<\gamma_{\!_f}\leq T} \Bigg| \sum_{p\le \sqrt{x} } \frac{a(p)}{p^{  2 \rho_{{\!_f}}+w}} \Bigg|^{2m} \ll m ! \ \! N_{\!_f}(T) \Bigg( \sum_{p \le \sqrt{x}} \frac{|a(p)|^2}{p^2} \Bigg)^{\! m}
  \end{equation}
where the implied constants depend only on $f$. 
\end{proposition}

The proofs of Proposition \ref{dmva} and Proposition \ref{dmvbd} rely on Montgomery and Vaughan's mean-value theorem for Dirichlet polynomials which we state in the following lemma.

\begin{lemma} \label{MontgomeryVaughan}
Let $\{a(n)\}$ and $\{b(n)\}$ be two sequences of complex numbers. Then, for any positive real numbers $T$ and $H$, we have 
\begin{equation}
  \label{eq:mv1}
   \int_{T}^{T+H}    \Bigg|   \sum_{n=1}^{\infty} \frac{a(n)}{n^{it}}  \Bigg|^2  dt = H \sum_{n=1}^{\infty} |a(n)|^2  + O\!\left( \sum_{n=1}^{\infty} n |a(n)|^2\right)
\end{equation}
and
\begin{equation}  \label{eq:mv2}
\begin{split}
   \int_{T}^{T+H}  
   \sum_{m=1}^{\infty}& \frac{a(m)}{m^{it}}   \sum_{n=1}^{\infty} \frac{b(n)}{n^{-it}} \ \!
    dt  = H \sum_{n=1}^{\infty} a(n) b(n) 
    + O\! \left( \sqrt{\sum_{n=1}^{\infty} n |a(n)|^2}   \sqrt{ \sum_{n=1}^{\infty} n |b(n)|^2 } \ \! \right).
\end{split}
\end{equation}
\end{lemma}
\begin{proof}
Equation (\ref{eq:mv1}) is due to Montgomery and Vaughan \cite[Corollary 3]{MV}. 
 Tsang \cite[Lemma 1]{Ts} derived the estimate in \eqref{eq:mv2} from \eqref{eq:mv1}.
\end{proof}

\smallskip

We now prove Propositions \ref{dmva}, \ref{dmvbd}, and \ref{dmv:primes}.

\begin{proof}[Proof of Proposition \ref{dmva}]
We begin by replacing $T$ and $2T$ by two numbers $\tau_1$ and $\tau_2$ which are $\gg (\log T)^{-1}$ from 
any ordinate $\gamma_{\!_f}$ of a non-trivial zero of $L(s,f)$. This allows us to use the estimates in Lemma \ref{LprimeLapprox} when estimating the horizontal portions of a certain contour integral arising below. By condition \eqref{eq:coefficientscondition1}, altering $T$ (respectively $2T$) by an amount which is $O(1)$ introduces an error of 
\begin{equation*}
\begin{split}
  &\ll \sum_{T \le \gamma_{\!_f} \le T+1} |A(\rho_{\!_f})|^2 \ll \log T  \max_{T \le \tau \le T+1} |A(\tfrac{1}{2}+i \tau)|^2 \ll (\log T) \Bigg( \sum_{n \le Y} \frac{|a(n)|}{\sqrt{n}} \Bigg)^2
   \ll T(\log T)^{3-2 \eta}. 
\end{split}
\end{equation*}
Therefore
$$ \sum_{T < \gamma_{\!_f} \leq 2T} |A(\rho_{\!_f})|^2 =\sum_{\tau_1 < \gamma_{\!_f} \leq \tau_2} |A(\rho_{\!_f})|^2+O\left(T(\log T)^{3-2 \eta}\right). $$
Now let $\mathscr{C}$ be the positively oriented rectangular contour with vertices $c+i\tau_1, c+i\tau_2, 1-c+i\tau_2$, and $1-c+i\tau_1$
where $c=1+\frac{1}{2 \log T}$.  By the calculus of residues, we have 
\[
  \sum_{\tau_1 < \gamma_{\!_f} \leq \tau_2} |A(\rho_{\!_f})|^2 = \frac{1}{2 \pi i} \int_{\mathscr{C}} \frac{L'}{L}(s,f) A(s)\overline{A}(1-s) ds:=I_1+I_2+I_3+I_4,
\]
say, where $I_1$ and $I_3$ denote the integrals along the vertical portions of $\mathscr{C}$, and $I_{2}$ and $I_{4}$ denote the integrals along the horizontal portions of  $\mathscr{C}$. 

\smallskip

 The integrals $I_1$ and $I_3$ contribute to the main term while $I_2$ and $I_4$ contribute to the error term. 
In fact, we shall prove that
\begin{equation} \label{eq:I1}
 I_1  =  -\frac{T}{2 \pi} \sum_{n \le Y}  \frac{(\Lambda_f* a)(n) \overline{a(n)}}{n} + O\!\left(  T(\log T)^{4-2\eta} + T \log T
   \sqrt{ \sum_{m=1}^{\infty} \frac{|(\Lambda_f*a)(m)|^2}{m^{1+\frac{1}{\log T}}}} \ \!\right),
\end{equation}
\begin{equation} \label{eq:I3}
I_{3}    =  \frac{T}{\pi} \log X \sum_{n \le Y} \frac{|a(n)|^2}{n} +\overline{I_1} + O\Big(  T (\log T)^{4-2\eta} \Big),
\end{equation}
and that
\begin{equation} \label{eq:I2I4}
 |I_2| + |I_4|  = O\Big(  T (\log T)^{4-2\eta} \Big). 
\end{equation}
Combining these three expressions yields the proposition. 

\smallskip

We first estimate $I_2$ and $I_4$, the  horizontal portions of the contour. Letting $\tau$ denote either $\tau_1$ or $\tau_2$, we claim that 
\begin{equation*} 
   \left| \int_{c+i \tau}^{1-c+i \tau} \frac{L'}{L}(s,f) A(s) \bar{A}(1\!-\!s) ds \right| = O\Big(  T (\log T)^{4-2\eta} \Big)
\end{equation*}
which implies  \eqref{eq:I2I4}. To prove the claim, notice that since $\tau$ is a real number satisfying the conditions of Lemma \ref{LprimeLapprox}, it follows that
\begin{equation} \label{int_simplify}
\begin{split}
 \bigg| \int_{c+i \tau}^{1-c+i \tau} \frac{L'}{L}(s,f) & A(s) \bar{A}(1\!-\!s) ds \bigg| 
 \\
 & \leq  \left\{ \max_{1-c\leq \sigma \leq c} \big| A(\sigma\!+\!i\tau) \bar{A}(1\!-\!\sigma\!-\!i\tau) \big|  \right\} \int_{1-c}^{c} \left| \frac{L'}{L}(\sigma+i\tau,f) \right| \ \! d\sigma 
\\
 & \leq  \left\{ \max_{1-c\leq \sigma \leq c} \big| A(\sigma\!+\!i\tau) \bar{A}(1\!-\!\sigma\!-\!i\tau) \big|  \right\} \int_{-1}^{2} \left| \frac{L'}{L}(\sigma+i\tau,f) \right| \ \! d\sigma 
\\
&\ll  \log T \left\{ \max_{1-c \leq \sigma \leq c} \big| A(\sigma\!+\!i\tau) \bar{A}(1\!-\!\sigma\!-\!i\tau) \big| \right\}
 \end{split}
\end{equation}
upon applying the  estimate for the integral in \eqref{log_deriv_int}. By our assumption \eqref{eq:coefficientscondition1} and partial summation, for $Y\asymp T$, we see that
\begin{equation*}
 \sum_{n\leq Y} \frac{|a(n)|}{n^\sigma} \ll \left\{ \begin{array}{ll}
T^{1-\sigma} (\log T)^{1-\eta}, &\mbox{ if $1-c\leq \sigma \leq \frac{1}{2}$,} \\
T^{1-\sigma}  (\log T)^{2-\eta}, &\mbox{ if $\frac{1}{2}<\sigma \leq c$.}
       \end{array} \right.
 \end{equation*}
Thus, estimating trivially we see that
\begin{equation} \label{trivial_est}
\max_{1-c \leq \sigma \leq c} \big| A(\sigma\!+\!i\tau) \bar{A}(1\!-\!\sigma\!-\!i\tau) \big| \ll T (\log T)^{3-2\eta}.
\end{equation}
Inserting the bound in \eqref{trivial_est} into the last line of \eqref{int_simplify}, we establish our claim and thus \eqref{eq:I2I4}.

\smallskip

We now estimate $I_1$. After a change of variables, we have
\begin{align*}
  I_1
  & = \frac{1}{2 \pi} \int_{\tau_1}^{\tau_2}  \frac{L'}{L}(c+it,f) A(c+it) \overline{A}(1\!-\!c\!-\!it) dt.
\end{align*}
Using condition \eqref{eq:coefficientscondition1}, we can replace $A(c+it)$ by the difference of an absolutely convergent Dirichlet series and its tail. Setting
$$A(c+it) = \sum_{m=1}^{\infty} \frac{a(m)}{m^{c+it}} - \sum_{m > Y} \frac{a(m)}{m^{c+it}}$$ 
we can decompose $I_1$ as $I_1  = I_{1}' - I_{1}''$ where 
\begin{equation*}
\begin{split}
    I_{1}'   &   = - \frac{1}{2 \pi}  \int_{\tau_1}^{\tau_2} \sum_{m=1}^{\infty} \frac{(\Lambda_f*a)(m)}{m^{c+it}} \sum_{n \le Y} \frac{\overline{a(n)}}{n^{1-c-it}} \ \! dt 
 \end{split}
\end{equation*}
  and
  \begin{equation*}
\begin{split}
  I_{1}'' &   =  -\frac{1}{2 \pi} 
  \int_{\tau_1}^{\tau_2} \sum_{k=1}^{\infty} \frac{\Lambda_{f}(k)}{k^{c+it}} \sum_{m > Y} \frac{a(m)}{m^{c+it}}
  \sum_{n \le Y} \frac{\overline{a(n)}}{n^{1-c-it}}  \ \!  dt.
\end{split}
\end{equation*}
This trick allows us to write the main-term in the proposition using a convolution of two arithmetic functions, $\Lambda_f*a$, which is essential to the utility of this proposition in the proof of Theorem \ref{2nd_moment}. Applying Lemma \ref{MontgomeryVaughan}, it follows that
\begin{equation*}
  \label{eq:IR1}
 I_{1}' =  - \frac{(\tau_2\!-\!\tau_1)}{2 \pi}  \sum_{n \le Y}  \frac{(\Lambda_f* a)(n) \overline{a(n)}}{n} 
 + O\!\left( 
 \sqrt{  
 \sum_{m=1}^{\infty} \frac{|(\Lambda_f*a)(m)|^2}{m^{2c-1}}}
 \sqrt{
 \sum_{m \le Y} \frac{|a(m)|^2}{m^{1-2c}}}  \ \!  \right).
\end{equation*}
Now, by condition \eqref{eq:coefficientscondition2}, we have
\begin{equation*}
  \label{eq:am2bound1}
  \sum_{m \le Y} \frac{|a(m)|^2}{m^{1-2c}} \ll \sum_{m \le Y} m |a(m)|^2 \ll T^2 (\log T)^2
\end{equation*}
since $Y\asymp T$ and, by Cauchy's inequality, we have
\begin{equation*}
\begin{split}
 \left| \sum_{n \le Y}  \frac{(\Lambda_f* a)(n) \overline{a(n)}}{n} \right| &\leq \sqrt{ \sum_{n\leq Y} \frac{|a(n)|^2}{n^{1-\frac{1}{\log T}}} } \sqrt{\sum_{n\leq Y} \frac{|(\Lambda_f*a)(n)|^2}{n^{1+\frac{1}{\log T}}} }
 \\
 &\ll \big( \log Y )^{3/2} \sqrt{ \sum_{n=1}^\infty \frac{|(\Lambda_f*a)(n)|^2}{n^{1+\frac{1}{\log T}}} }
\end{split}
\end{equation*}
where the last estimate follows from another application of  \eqref{eq:coefficientscondition2} and partial summation. Therefore, recalling  $c=1+\frac{1}{2 \log T}$, we derive that
\begin{equation}
  \label{eq:I1prime}
I_1' =  - \frac{T}{2 \pi} \sum_{n \le Y}  \frac{(\Lambda_f* a)(n) \overline{a(n)}}{n} 
 + O\!\left( T\log T \sqrt{ \sum_{n=1}^\infty \frac{|(\Lambda_f*a)(n)|^2}{n^{1+\frac{1}{\log T}}} }\ \! \right)
\end{equation}
since $Y\asymp T$ and $\tau_2 - \tau_1 = T +O(1)$.

\smallskip

The second integral $I_1''$ can be estimated directly, without appealing to Lemma \ref{MontgomeryVaughan}. Interchanging the sums and the integral, as justified by the absolute convergence of the series for $L'(s,f)/L(s,f)$, we have
\begin{equation*}
\begin{split}
 I_1''&= -\sum_{k=1}^{\infty}\sum_{m>Y}\sum_{n\leq Y} \frac{\Lambda_{f}(k)a(m) \overline{a(n)}}{k^{c}m^{c}n^{1-c}} 
 \int_{\tau_1}^{\tau_2} \left(\frac{n}{k m}\right)^{it} dt
 \\
& \ll \sum_{k=1}^{\infty}\sum_{m>Y}\sum_{n\leq Y} \frac{\Lambda_{f}(k)|a(m)| |a(n)|}{k^{c}m^{c}n^{1-c} |\log(n/k m)|}
 \end{split}
 \end{equation*}
after integrating term-by-term. The key point here is that the arithmetic function $\Lambda_{f}(\cdot)$ is supported on prime powers. Thus, $k m > 2Y$  and $|\log(n/k m)| > \log 2$. Since $|\Lambda_{f}(n)|\leq 2\Lambda(n)$, it follows that 
\begin{equation*}
\begin{split}
 I_{1}'' \ll \sum_{k=1}^{\infty} \frac{\Lambda(k)}{ k^{c}}\sum_{m>Y}\frac{|a(m)|}{m^{c}}\sum_{n\leq Y} \frac{|a(n)|}{n^{1-c}} \ll
 -\frac{\zeta'}{\zeta}(c) \sum_{m>Y}\frac{|a(m)|}{m^{c}}\sum_{n\leq Y} |a(n)| .
 \end{split}
 \end{equation*}
Since $c=1+\frac{1}{2}(\log T)^{-1}$, we see that $-\frac{\zeta'}{\zeta}(c) \ll \log T$. Moreover, by assumption  \eqref{eq:coefficientscondition1} and partial summation imply that 
$$\sum_{n\leq Y} |a(n)| \ll T(\log T)^{1-\eta} \quad \text{and} \quad  \sum_{m>Y}|a(m)|m^{-c} \ll (\log T)^{2-\eta}.$$
 It follows that $I_{1}'' \ll T(\log T)^{4-2\eta}$, and combining this estimate with \eqref{eq:I1prime} yields \eqref{eq:I1}.

\smallskip

We now evaluate $I_3$, the contribution for the left-hand side of the contour.  By the functional equation, we can write $I_3=I_{3}'-I_{3}''$  where
\begin{align*}
  I_{3}'  =  \frac{1}{2 \pi i} \int_{1-c+i\tau_2}^{1-c+i\tau_1}  \frac{\psi_{_f}'}{\psi_{_f}}(s) A(s) \overline{A}(1\!-\!s) ds  
    \end{align*}
  and
 \[
  I_{3}''  = \frac{1}{2 \pi i} \int_{1-c+i\tau_2}^{1-c+i\tau_1}  \frac{L'}{L}(1\!-\!s,\overline{f}) A(s) \overline{A}(1\!-\!s) ds.
 \] 
After a variable change, it follows that 
\begin{equation}
\begin{split}
  \label{eq:I3''identity}
  I_{3}''    & = \frac{1}{2 \pi} \int_{\tau_2}^{\tau_1}    \frac{L'}{L}(c\!-\!it,\overline{f}) A(1\!-\!c\!+\!it) \overline{A}(c\!-\!it)  dt  
  \\
    &= \overline{ 
   \frac{1}{2 \pi i} \int_{c+i\tau_2}^{c+i\tau_1}   
     \frac{L'}{L}(s,f)  \overline{A}(1\!-\!s)A(s) ds }  = -\overline{I_1}.
\end{split}
\end{equation}
We now estimate the integral $I_3'$. Shifting the line of integration right from $\Re(s) =1-c$ to $\Re(s) =\frac{1}{2}$ and using \eqref{trivial_est}, we see that
\begin{align*}
    I_{3}'  =- \frac{1}{2 \pi} \int_{\tau_1}^{\tau_2}    \frac{\psi_{_f}'}{\psi_{_f}}(\tfrac{1}{2}\!+\!it) \big| A(\tfrac{1}{2}\!+\!it)\big|^2 \ \! dt + O\Big(  T (\log T)^{4-2\eta} \Big).
\end{align*}
Applying the estimate for $\psi_{_f}'(s)/\psi_{_f}(s)$ in \eqref{eq:psiprimepsistirling}, the integral on the right-hand of the above equation is
\begin{align*}
 \frac{ \log X}{ \pi}   \int_{\tau_1}^{\tau_2} \big| A(\tfrac{1}{2}\!+\!it)\big|^2 \ \! dt
  + O\!\left(  \int_{\tau_1}^{\tau_2} \big| A(\tfrac{1}{2}\!+\!it)\big|^2 \ \! dt  \right). 
\end{align*}
Furthermore, Lemma \ref{MontgomeryVaughan} and \eqref{eq:coefficientscondition2} imply that
\[
   \int_{\tau_1}^{\tau_2} \big| A(\tfrac{1}{2}\!+\!it)\big|^2 \ \! dt =
   \sum_{n \le Y} \frac{|a(n)|^2}{n} ( \tau_2-\tau_1 + O(n))
   = T  \sum_{n \le Y} \frac{|a(n)|^2}{n} + O( T (\log T)^2). 
\]
Hence, we deduce that 
\begin{align*}
 I_{3}' 
  &=  \frac{T}{\pi} \log X \sum_{n \le Y} \frac{|a(n)|^2}{n} + O(  T (\log T)^{4-2\eta} )
\end{align*}
since $\eta \leq \frac{1}{2}$.  Combining this estimate with \eqref{eq:I3''identity} yields \eqref{eq:I3}, and completes the proof of the proposition. 
\end{proof}

\begin{proof}[Proof of Proposition \ref{dmvbd}]  Under the assumption of  $\text{RH}_f$,  the proof of this proposition follows from Lemma \ref{NfTSfT} and Lemma \ref{MontgomeryVaughan}. 
Recall that 
\[
     S_1 = \sum_{n=1}^{\infty} |b(n)|^2   \ \! (T + n ) \ \!  , \quad 
  S_2 = \sum_{n=1}^{\infty} |b(n) |^2  (T + n) \log^2 n, \quad \text{and} \quad 
  S_3 = \sum_{ n=1}^{\infty} |b(n)|. 
\]
We choose $\tau_1$ and $\tau_2$ such that $T-1 \le \tau_1 \le T$ and $2T \le \tau_2 \le 2T+1$ where $\tau_j$ is chosen not coincide with any $\gamma_{\!_f}$ for $j=1$ or 2. 
Summing by parts, we have
\begin{align*}
   \sum_{T < \gamma_{\!_f} \le 2T} | B(\gamma_{\!_f})|^2\ & \le   \sum_{\tau_1 < \gamma_{\!_f} \le \tau_2} | B(\gamma_{\!_f})|^2  
   \\
     & = \int_{\tau_1}^{\tau_2} |B(t)|^2 \ \! d N_{\!_f}(t) 
     \\
     &  = \int_{\tau_1}^{\tau_2} |B(t)|^2 \ \! \vartheta'_{\!_f}(t) \ \! dt + \int_{\tau_1^-}^{\tau_2} |B(t)|^2 \ \! dS_{\!_f}(t).
\end{align*}
Using the estimate $\vartheta'_{\!_f}(t) = O(\log t)$, it follows that
the first integral in the last line of the inequality is 
\begin{equation}
   \label{eq:J1}
   \ll  \log T \int_{\tau_1}^{\tau_2} |B(t)|^2dt.
\end{equation}
Now set $\overline{B}(t) = \sum_{n=1}^{\infty} \overline{b(n)} n^{-it}$. 
To estimate the second integral in the last line of the above inequality, we integrate by parts and use the bound for $S_{\!_f}(t)$ in \eqref{eq:SfTbound}. In this way, we find that 
\begin{equation}
\begin{split}
  \label{eq:J2}
  \int_{\tau_1^-}^{\tau_2} |B(t)|^2 \ \! dS_{\!_f}(t) & = \int_{\tau_1^-}^{\tau_2} B(t) \overline{B}(-t) \ \! dS_{\!_f}(t) 
   \\
  & = S_{\!_f}(t) |B(t)|^2 \Big|_{\tau_1^-}^{\tau_2} - \int_{\tau_1}^{\tau_2} S_{\!_f}(t) B(t) \overline{B}'\!(-t) \ \! dt 
  + \int_{\tau_1}^{\tau_2} S_{\!_f}(t)  B'(t) \overline{B}(-t) \ \! dt 
  \\
  & \ll \frac{\log T}{\log \log T}  \left( \max_{\tau_1 \le t \le \tau_2} |B(t)|^2
  + \int_{\tau_1}^{\tau_2} \Big( | B(t) \overline{B}'(t)| + |B'(t) \overline{B}(t)| \Big) \ \! dt  \right) 
  \\
  & \ll \frac{\log T}{\log \log T} 
  \left( 
   \max_{\tau_1 \le t \le \tau_2} |B(t)|^2+ \sqrt{ \int_{\tau_1}^{\tau_2} |B(t)|^2 \ \! dt  \cdot \int_{\tau_1}^{\tau_2} |B'(t)|^2 dt .}
  \ \! \right)
  \\
  & \ll \frac{\log T}{\log \log T} 
  \left( 
   \max_{\tau_1 \le t \le \tau_2} |B(t)|^2+ \sqrt{ \int_{T/2}^{4T} |B(t)|^2 \ \! dt  \cdot \int_{T/2}^{4T} |B'(t)|^2 dt .}
  \ \! \right).
\end{split}
\end{equation}
Note that the inequality in the penultimate line follows from an application of the Cauchy-Schwartz inequality to the integral in the previous line. Finally, we observe that 
\begin{align*}
    \int_{T/2}^{4T} |B(t)|^2 dt \ll S_1,  \quad
     \int_{T/2}^{4T} |B'(t)|^2 dt \ll S_2,  \quad \text{and} \quad
     |B(t)| \le \sum_{n=1}^{\infty} |b(n)|  = S_3.   
\end{align*}
The first two estimates follow from Lemma \ref{MontgomeryVaughan} and the third estimate is trivial.
Inserting these bounds into \eqref{eq:J1} and \eqref{eq:J2}, we complete the proof of the proposition.
\end{proof}

\begin{proof}[Proof of Proposition \ref{dmv:primes}] The proof of this lemma uses the Landau-Gonek explicit formula (Lemma \ref{gonek}). The estimate in \eqref{1st_prime_sum}  is a discrete analogue of Lemma 3 in Soundararajan \cite{S} and, as might be expected, the proof of the two estimates in \eqref{1st_prime_sum} and \eqref{2nd_prime_sum} are similar.  

\smallskip

Define the coefficients $a_{m,x}(n)$ by
\begin{equation}\label{newmvt1}
\Bigg(\sum_{p\le x} \frac{a(p)}{p^s}\Bigg)^{\!m} = \sum_{n\le x^m} \frac{a_{m,x}(n)}{n^s},
\end{equation}
where $a_{m,x}(n)=0$ unless $n$ is the product of $m$ (not necessarily distinct) primes $p\le x$. With this notation, we have
\begin{equation}\label{newmvt2}
 \sum_{n\le x^m} \frac{|a_{m,x}(n)|^2}{n} \leq m ! \Bigg(\sum_{p\le x} \frac{|a(p)|^2}{p}\Bigg)^{\! m} \quad \text{and} \quad  \sum_{n\le x^{m/2}} \frac{|a_{m,\sqrt{x}}(n)|^2}{n^2} \leq m ! \Bigg(\sum_{p\le \sqrt{x}} \frac{|a(p)|^2}{p^2}\Bigg)^{\! m}.
\end{equation}
The first inequality is due to Soundararajan (see the proof of Lemma 3 in \cite{S}), and the second inequality can be proved similarly.
 
\smallskip

Assuming $\text{RH}_f$, we note that $1-\rho_{\!_f} = \overline{\rho_{\!_f}}$ for any non-trivial zero $\rho_{\!_f} = \frac{1}{2}+i\gamma_{\!_f}$ of $L(s,f)$. This observation, combined with \eqref{newmvt1}, implies that
\[
 \Bigg| \sum_{p\le x} \frac{a(p)}{p^{  \rho_{{\!_f}}+w}} \Bigg|^{2 m} = \sum_{d\leq x^m} \sum_{n\leq x^m} \frac{a_{ m,x}(d) \overline{a_{ m,x}(n)} }{d^{\rho_{\!_f}+w} n^{1-\rho_{\!_f}+\overline{w}}}
\]
and, moreover, that
\begin{equation*}
\begin{split}
  \sum_{0<\gamma_{\!_f}\leq T} \Bigg| \sum_{p\le x} \frac{a(p)}{p^{  \rho_{{\!_f}}+w}} \Bigg|^{2 m} = N_{\!_f}(T) & \sum_{n\le x^m}  \frac{|a_{ m,x}(n)|^2}{n^{1+2\Re(w)}} 
  \\
  &+ 2\Re   \sum_{d\leq x^ m} \frac{a_{ m,x}(d)}{d^w} \sum_{d<n\leq x^m} \frac{ \overline{a_{m,x}(n)} }{ n^{1+\overline{w}}}   \sum_{0<\gamma_{\!_f}\leq T} \Big(\frac{n}{d}\Big)^{  \rho_{{\!_f}}}.
\end{split}
\end{equation*}
Thus, by \eqref{newmvt2} and the fact that $\Re(w) \ge 0$, it follows that
\[
N_{\!_f}(T) \sum_{n\le x^m}  \frac{|a_{m,x}(n)|^2}{n^{1+2\Re(w)}} \ll N_{\!_f}(T)  \sum_{n\le x^m}  \frac{|a_{m,x}(n)|^2}{n} \ll m ! \ \! N_{\!_f}(T)\Bigg(\sum_{p\le x} \frac{|a(p)|^2}{p}\Bigg)^{\! m}.
\]
Appealing to Lemma \ref{gonek}, using the estimate $|\Lambda_f(n)|\le 2 \Lambda(n)$ and the fact that $\Re(w) \ge 0$, it follows that
\[
\Bigg|\sum_{d\leq x^m} \frac{a_{m,x}(d)}{d^w} \sum_{d<n\leq x^m} \frac{ \overline{a_{m,x}(n)} }{ n^{1+\overline{w}}}   \sum_{0<\gamma_{\!_f}\leq T} \Big(\frac{n}{d}\Big)^{  \rho_{{\!_f}}} \Bigg| \le \Sigma_1+\Sigma_2+\Sigma_3+\Sigma_4
\]
where
\[
\Sigma_1 = \frac{T}{2\pi} \sum_{  d k \le x^m} \frac{|a_{m,x}(dk)||a_{m,x}(d)| \Lambda(k)}{dk},
\]
\[
\Sigma_2 = O\left( \log T \log \log T \sum_{d\leq x^m} \sum_{d<n\leq x^m} \frac{|a_{m,x}(d)a_{m,x}(n)|}{d}  \right),
\]
\[
\Sigma_3 = O\left(  \sum_{d\leq x^m}  \sum_{d<n\leq x^m} \frac{|a_{m,x}(d)  a_{m,x}(n)|}{d}\frac{\log(n/d)}{\langle n/d \rangle}   \right), 
\]
and
\[
\Sigma_4 = O\left( \log T \sum_{d\leq x^m}  \sum_{d<n\leq x^m} \frac{|a_{m,x}(d) a_{m,x}(n)|}{n \log(n/d)}  \right).
\]
In order to bound these sums, we shall repeatedly use the inequality
\begin{equation}\label{newmvt3}
2|a_{m,x}(d)a_{m,x}(n)| \le \Delta |a_{m,x}(d)|^2 + \frac{|a_{m,x}(n)|^2}{\Delta}
\end{equation}
which holds for any $\Delta >0$. We estimate $\Sigma_1$ first. 
By \eqref{newmvt3} with $\Delta=1$, we have
\begin{eqnarray*}
\Sigma_1 &\leq& \frac{T}{4\pi} \sum_{n\le x^{m}} \frac{|a_{m,x}(n)|^2}{n} \sum_{k|n} \Lambda(k) + \frac{T}{4\pi} \sum_{d\le x^{m}} \frac{|a_{m,x}(d)|^2}{d} \sum_{k \le x^m/d} \frac{\Lambda(k)}{k}
\\
&\ll& T \log(x^m) \sum_{n\leq x^m}  \frac{|a_{m,x}(n)|^2}{n}  \ll  m !  \ \! N_{\!_f}(T)  \Bigg(\sum_{p\le x} \frac{|a(p)|^2}{p}\Bigg)^{\! m}
\end{eqnarray*}
upon using \eqref{newmvt2}, the inequality $N_{\!_f}(T)\gg T\log T$, and the well known estimates $\sum_{k|n} \Lambda(k) = \log n$ and $ \sum_{k\leq \xi} \frac{\Lambda(k)}{k} \ll \log \xi$.

\smallskip

Next we bound $\Sigma_2$.  By \eqref{newmvt3} with $\Delta=1$, it follows that
\begin{eqnarray*}
\Sigma_2 &\ll& \log T \log\log T  \sum_{d\leq x^m} \sum_{d<n\leq x^m} \left\{ \frac{|a_{m,x}(d)|^2}{d} + \frac{|a_{m,x}(n)|^2 }{d} \right\}
\\
&\ll&  \log T \log\log T  \left\{  x^m \sum_{d\leq x^m}  \frac{|a_{m,x}(d)|^2 }{d} + \sum_{d\leq x^m} \frac{1}{d} \sum_{d<n\leq x^m} \frac{ x^m }{n}   |a_{m,x}(n)|^2 \right\}
\\
&\ll& T^{2/3} \log^2 T  \log\log T \sum_{n\leq x^m}  \frac{|a_{m,x}(n)|^2 }{n}
\\
&\ll& m !  \ \! N_{\!_f}(T)  \Bigg(\sum_{p\le x} \frac{|a(p)|^2}{p}\Bigg)^{\! m}
\end{eqnarray*}
using $x^m\le T^{2/3}$ and \eqref{newmvt2}. 

\smallskip

We now estimate $\Sigma_3$. Again using \eqref{newmvt3}, we have
\[
\Sigma_3 \ll 
\sum_{d\leq x^m} \sum_{d<n\leq x^m} \left\{ \frac{\Delta |a_{m,x}(d)|^2}{d} + \frac{|a_{m,x}(n)|^2 }{d \Delta} \right\} \frac{\log(n/d)}{\langle n/d \rangle}
 \ll \Sigma_{31}+\Sigma_{32},
\]
say, where $\Sigma_{31}$ is the double sum including the terms with $|a_{m,x}(d)|^2$ and  $\Sigma_{32}$ is the double sum including the terms with $|a_{m,x}(n)|^2$. 
 Writing $n=ad+b$ with $-d/2<b \le d/2,$ we see that
\[
\Sigma_{31} \ll  \sum_{d\leq x^m}  \frac{|a_{m,x}(d)|^2}{d} \sum_{a\le \frac{ x^m}{d}+1} \sum_{-\frac{d}{2}<b\le \frac{d}{2}} \frac{\log(a\!+\!b/d)}{\langle a\!+\! b/d \rangle}.
\]
Now $\langle a+b/d \rangle$ equals $|b|/d$ if $a$ is the power of a prime and  $\langle a+b/d \rangle \ge 1/2$, otherwise. Using  \eqref{newmvt2},  the estimate $\sum_{n\le x} \Lambda(n) \ll x$, and the inequality $x^m\le T^{2/3}$,  we find that
\begin{equation}
\begin{split}
  \label{Sigma31bd}
\Sigma_{31} &\ll \Delta \sum_{d\leq x^m}  \frac{|a_{m,x}(d)|^2}{d} \left\{  \sum_{a\le \frac{ x^m}{d}+1} \Lambda(a) \sum_{ 1<b\le \frac{d}{2}} \frac{d}{b}  + \sum_{a\le \frac{ x^m}{d}+1} \log(a\!+\!1) \sum_{ 1<b\le \frac{d}{2}} 1 \right\}
\\
&\ll \Delta x^m \log T \sum_{d\leq x^m}  \frac{|a_{m,x}(d)|^2}{d}   \ll   m! \ \! \Delta T^{2/3} \log T  \Bigg(\sum_{p\le x} \frac{|a(p)|^2}{p}\Bigg)^{\! m}.
\end{split}
\end{equation}
We now bound $\Sigma_{32}$. For $n\ge 2$, we define
\[
g(n)= \sum_{1 \le d \le n-1} \frac{1}{d \ \! \langle n/d \rangle}
\]
and since $ \langle n/d \rangle \ge \frac{1}{d}$ we note that $g(n) \le n$. 
Therefore, by \eqref{newmvt2}, we have
\begin{equation}
  \label{Sigma32bd2}
  \begin{split}
  \Sigma_{32} &\ll  \frac{1}{\Delta} \sum_{n\leq x^m} |a_{m,x}(n)|^2  \sum_{d<n} \frac{\log(n/d)}{d \ \! \langle n/d \rangle} \le \frac{1}{\Delta} \log(x^m) \sum_{n\leq x^m} |a_{m,x}(n)|^2 g(n)
\\
&\le \frac{(x^m)^2}{\Delta} \log T   \sum_{n\leq x^m} \frac{|a_{m,x}(n)|^2}{n} \ll m! \, \frac{T^{4/3}}{\Delta} \log T  \Bigg(\sum_{p\le x} \frac{|a(p)|^2}{p}\Bigg)^{\! m}.
\end{split}
\end{equation}
Choosing $\Delta = T^{1/3}$ in \eqref{Sigma31bd} and \eqref{Sigma32bd2} yields
\[
  \Sigma_3 \ll  m! \ \!  N_{\!_f}(T)   \Bigg(\sum_{p\le x} \frac{|a(p)|^2}{p}\Bigg)^{\! m}.
\]

\smallskip

Finally, we require an estimate for $\Sigma_4$. Since $d<n$, we note that $n > \sqrt{dn}$ and thus
\begin{eqnarray*}
\Sigma_4 &\ll&  \log T \sum_{d\leq x^m}  \sum_{d<n\leq x^m} \frac{|a_{m,x}(d) a_{m,x}(n)|}{n \log(n/d)} \ll  \log T \sum_{d\leq x^m}  \sum_{d<n\leq x^m} \frac{|a_{m,x}(d) a_{m,x}(n)|}{\sqrt{dn} \log(n/d)} 
\\
&\ll&  \log T \sum_{d\leq x^m}  \sum_{d<n\leq x^m} \left\{ \frac{|a_{m,x}(d)|^2}{d} +\frac{ |a_{m,x}(n)|^2}{n} \right\} \frac{1}{ \log(n/d)}
\\
&\ll& \log T \sum_{d\leq x^m} \frac{|a_{m,x}(d)|^2}{d}   \sum_{ \substack{n\leq x^m \\ d\neq n}}  \frac{1}{|\log(n/d)|} \ \ll \ x^m \log(x^m) \log T  \sum_{d\leq x^m} \frac{|a_{m,x}(d)|^2}{d}  
\\
& \ll&  m !  \ \! N_{\!_f}(T)  \Bigg(\sum_{p\le x} \frac{|a(p)|^2}{p}\Bigg)^{\!m}
\end{eqnarray*}
again using $x^m\le T^{2/3}$ and \eqref{newmvt2}. Combining estimates, we complete the proof of \eqref{1st_prime_sum}.

\smallskip

The proof of \eqref{2nd_prime_sum} is similar. Again using the observation that $\text{RH}_f$ implies $1-\rho_{\!_f} = \overline{\rho_{\!_f}}$, it follows from \eqref{newmvt1} that
\begin{equation*}
\begin{split}
  \sum_{0<\gamma_{\!_f}\leq T} \Bigg| \sum_{p\le \sqrt{x}} \frac{a(p)}{p^{ 2\rho_{{\!_f}}+w}} \Bigg|^{2 m} &= N_{\!_f}(T) \sum_{n\le x^{m/2}}  \frac{|a_{ m,\sqrt{x}}(n)|^2}{n^{2+2\Re(w)}} 
  \\
  &\quad + 2\Re   \sum_{d\leq x^{m/2}} \frac{a_{ m,\sqrt{x}}(d)}{d^w} \sum_{d<n\leq x^{m/2}} \frac{ \overline{a_{m,\sqrt{x}}(n)} }{ n^{2+\overline{w}}}   \sum_{0<\gamma_{\!_f}\leq T} \Big(\frac{n^2}{d^2}\Big)^{  \rho_{{\!_f}}}.
\end{split}
\end{equation*}
By the second inequality in \eqref{newmvt2} and the fact that $\Re(w)>0$, we see that
\[
N_{\!_f}(T) \sum_{n\le x^{m/2}}  \frac{|a_{ m,\sqrt{x}}(n)|^2}{n^{2+2\Re(w)}} \ll N_{\!_f}(T) \sum_{n\le x^{m/2}}  \frac{|a_{ m,\sqrt{x}}(n)|^2}{n^{2}} \ll m ! \ \! N_{\!_f}(T) \Bigg( \sum_{p \le \sqrt{x}} \frac{|a(p)|^2}{p^2} \Bigg)^{\! m}.
\]
Using Lemma \ref{gonek} along with the fact that $|\Lambda_f(n^2)| \leq 2 \Lambda(n^2) = 2 \Lambda(n)$, splitting the sum into four parts as in the proof of  \eqref{1st_prime_sum} and estimating similarly, we can show that 
\[
\left|  \sum_{d\leq x^{m/2}} \frac{a_{ m,\sqrt{x}}(d)}{d^w} \sum_{d<n\leq x^{m/2}} \frac{ \overline{a_{m,\sqrt{x}}(n)} }{ n^{2+\overline{w}}}   \sum_{0<\gamma_{\!_f}\leq T} \Big(\frac{n^2}{d^2}\Big)^{  \rho_{{\!_f}}} \right|  \ll m ! \ \! N_{\!_f}(T)  \Bigg( \sum_{p \le \sqrt{x}} \frac{|a(p)|^2}{p^2} \Bigg)^{\! m},
\]
which is the estimate that we require. 
\end{proof}

\section{Averages of the Fourier coefficients $\lambda_{f}(n)$}

In this section we derive asymptotic formulae and estimates for certain sums involving the arithmetic functions $\lambda_f(n)$ and $\Lambda_f(n)$.
These formulae are required when we apply Proposition \ref{dmva} and Proposition \ref{dmvbd} 
to certain Dirichlet polynomials with coefficients involving $\lambda_f(n)$.

\begin{proposition} \label{sums}
 Let $f\in H_k(q,\chi)$. Then, for sufficiently large $x$, 
 \begin{align}
 \label{eq:ransel}
 \mathscr{A}_f(x) & := \sum_{n\leq x} |\lambda_{f}(n)|^{2} = c_{f} \ \! x + O( x^{3/5}),  \\
   \label{eq:prime1}
   \mathscr{B}_f(x) &  := \sum_{p \le x} \frac{|\lambda_f(p)|^2 \log p}{p} =  \log x + O(1), \\
     \label{eq:prime2}
   \mathscr{C}_f(x) &  := \sum_{p \le x} \frac{|\lambda_f(p)|^2}{p} =  \log \log x + O(1), \ \text{ and} \\
  \label{eq:absval}
    \mathscr{D}_f(x) &  := \sum_{n \le x} |\lambda_f(n)|  \ll x (\log x)^{-\delta}
\end{align}
where $c_f$ is the positive constant defined in \eqref{eq:cfdefinition} and 
\begin{equation}
  \label{eq:delta}
\delta = \tfrac{4}{5}-\tfrac{1}{5} \sqrt{ \tfrac{27}{2} }= 0.065153\ldots .
\end{equation}
\end{proposition}

\begin{proof}  All of these estimates follow from well known results. The estimate in \eqref{eq:ransel} can be proved using the techniques of Rankin \cite{Ran1} and Selberg \cite{Se3}. Proposition 2.3 of Rudnick and Sarnak \cite{RS} asserts that 
\[ \sum_{p\leq x} \frac{|\lambda_{f}(p)|^{2}\log^{2} p}{p} = \frac{1}{2}\log^{2} x + O(\log x),
\] 
from which \eqref{eq:prime1} and \eqref{eq:prime2} now follow by partial summation. Finally, the estimate in \eqref{eq:absval} is due to Rankin \cite{Ran} when $q=1$ and was extended by Shahidi \cite{Sh2} to higher levels ($q>1$) when $\chi$ is trivial. The Rankin-Shahidi proof can be extended to all $f\in H_k(q,\chi)$ using the estimate in \eqref{eq:prime2} along with the asymptotic formula
\[
\sum_{p \le x} \frac{|\lambda_f(p)|^4}{p} =2\log\log x+O(1)
\]
which can be deduced from the work of Kim and Shahidi \cite{KimShahidi} and of Kim \cite{Kim} on symmetric fourth power $L$-functions.
\end{proof}

The main term in Proposition \ref{dmva} involved sums of the form
\[
  \sum_{n \le x} \frac{|a(n)|^2}{n}, \quad 
  \sum_{n \le x}  \frac{(\Lambda_f* a)(n) \overline{a(n)}}{n}, \quad \text{and}  \quad
  \sum_{n=1}^{\infty} \frac{|(\Lambda_f*a)(n)|^2}{n^{\sigma}}
\]
where $a$ is a complex-valued arithmetic function satisfying certain average growth conditions, $*$ denotes Dirichlet convolution, and $\sigma > 1$.  In the following lemmas, we evaluate such sums in the cases where
$a(n)=\alpha_f(n)$ and  $a(n)=\beta_{f,x}(n)$, the arithmetic functions defined in \eqref{eq:defncoeffs}. 

\smallskip

\begin{lemma} 
   \label{rankinselbergpartialsum}
Let $f\in H_k(q,\chi)$.  Then 
$$ \sum_{n\leq x} \frac{|\alpha_{f}(n)|^{2}}{n}=\frac{1}{3}c_{f} \log^3  x + O(\log^{2}x )$$ and
$$\sum_{n\leq x} \frac{|\beta_{f,x}(n)|^{2}}{n}=\frac{7}{3}c_{f} \log^3 x + O(\log^{2}x )$$
where $c_{f}$ is the positive constant in \eqref{eq:cfdefinition}.
\end{lemma}

\begin{proof}
By partial summation and \eqref{eq:ransel}, we have
\begin{equation*}
\begin{split}
  \sum_{n\leq x} \frac{|\alpha_{f}(n)|^{2}}{n} &= \int_{1^{-}}^{x} \frac{(\log t)^2}{t} d \mathscr{A}_f(t) 
  \\
  &=  c_f \int_{1}^{x} \frac{(\log t)^2}{t} dt + O(\log^2 x ) = \frac{1}{3} c_f \log^3 x + O(\log^2 x). 
\end{split}
\end{equation*}
Similarly, we see that 
\begin{equation*}
\begin{split}
 \sum_{n\leq x} \frac{|\beta_{f,x}(n)|^{2}}{n}  &=   \int_{1^{-}}^{x} \frac{ \log(\frac{x^2}{t})^2}{t} d \mathscr{A}_f(t)
 \\
 &=   c_f \int_{1}^{x} \frac{ \log(\frac{x^2}{t})^2}{t} dt + O(\log^2 x)
 =  \frac{7}{3} c_f \log^3 x + O(\log^2 x).  \\
 \end{split}
 \end{equation*}
This completes the proof of the lemma.
\end{proof}

\smallskip

\begin{lemma} \label{RSBound3}
Let $f\in H_k(q,\chi)$. Then, for any prime $p$, we have
$$ \sum_{m\leq y} \frac{|\lambda_{f}(mp)|^{2}}{m} = O(\log y)$$
where the implied constant is independent of $p$.
\end{lemma}

\begin{proof}
By the multiplicativity of the Fourier coefficients $\lambda_f(n)$,  we have
\begin{equation*}
\begin{split}
\sum_{m\leq y} \frac{|\lambda_{f}(mp)|^{2}}{m} &= \sum_{\substack{m\leq y \\ (p,m)=1}} \frac{|\lambda_{f}(m)|^{2} |\lambda_{f}(p)|^{2}}{m}+ 
\sum_{j \geq 1} \sum_{\substack{n\leq y/p^{j} \\ (p,n)=1}} \frac{|\lambda_{f}(n)|^{2} |\lambda_{f}(p^{j+1})|^{2}}{np^{j}} 
\\
&\leq \sum_{m\leq y} \frac{|\lambda_{f}(m)|^{2}}{m} \left( 4 + \sum_{j=1}^{\infty} \frac{(j\!+\!2)^{2}}{2^{j}}\right) 
\end{split}
 \end{equation*}
using Deligne's estimate $|\lambda_{f}(n)| \leq d(n)$.  Since the sum over $j$ converges, the lemma now follows from \eqref{eq:ransel} and partial summation.
 \end{proof}

\smallskip

\begin{lemma}\label{RSBound4}
Let $f\in H_k(q,\chi)$, $\ell$ be a natural number, and $p^\ell \le x$. Then, for $1 \le y \le x$, we have
\begin{equation*}
\begin{split}
\sum_{n\leq y} \frac{\alpha_{f}(n) \overline{\alpha_{f}(p^{\ell}n)}}{n} &= \sum_{n\leq y} \frac{|\lambda_{f}(n)|^{2} \overline{\lambda_{f}(p^{\ell})}}{n} \log(n) \log(p^{\ell}n) +O\!\left( \frac{\ell  \log^3 x}{p} \right)
\end{split}
\end{equation*}
and
\begin{equation*}
\begin{split}
\sum_{n\leq y} \frac{\beta_{f,x}(n) \overline{\beta_{f,x}(p^{\ell}n)}}{n} &= \sum_{n\leq y} \frac{|\lambda_{f}(n)|^{2} \overline{\lambda_{f}(p^{\ell})}}{n} \log\Big(\frac{x^{2}}{n}\Big) \log\Big(\frac{x^{2}}{p^{\ell}n}\Big)  +O\!\left( \frac{\ell  \log^{3} x}{p} \right).
\end{split}
\end{equation*}
\end{lemma}

\begin{proof}
Both of the estimates in the above lemma can be established in a very similar manner. For this reason, we prove only the first estimate involving $\alpha_{f}(n)$. First we partition the sum into disjoint pieces based upon the exact power of $p$ dividing $n$. In particular, we have
\begin{equation}
  \label{eq:decomposition}
  \sum_{n\leq y} \frac{\alpha_{f}(n) \overline{\alpha_{f}(p^{\ell}n)}}{n} = \sum_{\substack{n\leq y \\ (p,n)=1}}  \frac{|\lambda_{f}(n)|^2 \overline{\lambda_{f}(p^{\ell})}}{n}\log(n) \log(p^{\ell}n) +  
  \sum_{j \geq 1} \sum_{\substack{n\leq y \\ p^{j} || n}}  \frac{\alpha_{f}(n) \overline{\alpha_{f}(p^{\ell}n)}}{n}.
\end{equation}
Using the inequality $|\lambda_{f}(n)|\leq d(n)$, the second sum on the right-hand side is 
\begin{equation*}
\begin{split}
&\ll \sum_{j\geq 1} \sum_{m\leq y/p^{j}}  \frac{|\lambda_{f}(m)|^{2} d(p^{j})d(p^{\ell+j})}{p^{j}m}\log(p^{j}m)\log(p^{\ell+j}m) 
\\
 & \ll  \log y\log(p^{\ell}y) \sum_{j=1}^{\infty} \frac{(j\!+\!1)(\ell\!+\!j\!+\!1)}{p^{\ell}}  \sum_{m\leq y} \frac{|\lambda_{f}(m)|^{2}}{m}  
\\
& \ll  \frac{\ell \log^3 x}{p}. \\
\end{split}
\end{equation*}
Those terms with $n \le y$ and $p|n$ may be added into the first sum on the right-hand side of \eqref{eq:decomposition} with an error of
\begin{equation*}
\begin{split}
\ll \sum_{\substack{n\leq y \\ p|n}} \frac{|\lambda_{f}(n)|^{2} |\lambda_f(p^\ell)|}{n} \log(n) \log(p^{\ell}n) & \ll 
\frac{\ell}{p} \log y \log(p^{\ell}y)  \sum_{ m\leq y/p } \frac{ |\lambda_{f}(mp)|^{2}}{m} 
 \ll  \frac{\ell \log^3 x}{p}
\end{split}
\end{equation*}
using $|\lambda_f(p^\ell)|\le 2\ell$ and Lemma \ref{RSBound3}.
By combining estimates, we complete the proof of the lemma.
\end{proof}

\begin{lemma}\label{lambda_prime}
 Let $f\in H_k(q,\chi)$. Then
\begin{equation} \label{first}
\sum_{mn\leq x} \frac{\Lambda_{f}(m)\alpha_f(n) \overline{\alpha_f(mn)}}{mn} = \frac{1}{8} c_{f}  \log^{4}x + O(\log^{3}x)
\end{equation}
and 
\begin{equation} \label{second}
  \sum_{mn\leq x} \frac{\Lambda_{f}(m)\beta_{f,x}(n) \overline{\beta_{f,x}(mn)}}{mn} = \frac{9}{8} c_{f}  \log^{4}x + O(\log^{3}x)
\end{equation}
where $c_{f}$ is the positive constant defined in \eqref{eq:cfdefinition}.
\end{lemma}

\begin{proof} 
The sum we are evaluating in \eqref{first} is 
\[
  \sum_{p^{\ell}\leq x} \frac{\Lambda_{f}(p^{\ell})}{p^{\ell}} \sum_{n\leq x/p^{\ell}} \frac{\alpha_f(n) \overline{\alpha_f(p^{\ell}n)}}{n}. 
\]
By Lemma \ref{RSBound4} and the inequality $|\Lambda_{f}(p^{\ell})| \leq 2 \log p$, it equals
\begin{equation*}
\begin{split}
\sum_{p^{\ell}\leq x} \frac{\Lambda_{f}(p^{\ell})}{p^{\ell}} \sum_{n\leq x/p^{\ell}} \frac{|\lambda_{f}(n)|^{2} \overline{\lambda_{f}(p^{\ell})}}{n} \log(n) \log(p^{\ell}n) 
 +O\!\left(  \log^{3} x \sum_{p^{\ell}\leq x} \frac{k \log p }{p^{\ell+1}}  \right).
 \end{split}
  \end{equation*}
The error term here is $O(\log^{3} x)$.  Using the estimates $|\Lambda_{f}(p^{\ell})|\leq 2 \log p$ and $|\lambda_{f}(p^{\ell})|\leq \ell+1$, 
it can be shown that the prime powers $p^{\ell}$ with $\ell \ge 2$ may be removed from the above sum over $n$ with an error of $O( \log^3 x )$. 
Since $\Lambda_{f}(p)=\lambda_{f}(p)\log p$, it follows that 
  \begin{equation*}
\begin{split}
\sum_{mn\leq x} \frac{\Lambda_{f}(m)\alpha_f(n) \overline{\alpha_f(mn)}}{mn} &= \sum_{p\leq x} \frac{\Lambda_{f}(p)}{p} \sum_{n\leq x/p} \frac{|\lambda_{f}(n)|^{2} \overline{\lambda_{f}(p)}}{n} \log(n) \log(pn) + O(\log^3 x)
\\
& = \sum_{p\leq x} \frac{|\lambda_{f}(p)|^2 \log p}{p} \sum_{n\leq x/p} \frac{|\lambda_{f}(n)|^{2}}{n} \log(n) \log(pn) + O(\log^3 x).
\end{split}
\end{equation*}
By two applications of Stieltjes integration along with the estimates in \eqref{eq:ransel} and \eqref{eq:prime1},  the sum on the right-hand side of the above expression equals 
\begin{align*}
   \sum_{p \le x}  \frac{ |\lambda_f(p)|^2 \log p}{p} & \int_{1^{-}}^{\frac{x}{p}} \frac{(\log t) (\log p + \log t)}{t} d\mathscr{A}_f(t)  \\
  & =  c_f   \sum_{p \le x} \frac{ |\lambda_f(p)|^2 \log p}{p} \left\{ \frac{1}{2} \log^2(\tfrac{x}{p}) \log p + \frac{1}{3} \log^3(\tfrac{x}{p}) \right\}  +O(\log^3 x)
  \\
  & = c_f \int_{1^{-}}^x \left\{ \frac{1}{2} \log^2(\tfrac{x}{u}) \log u + \frac{1}{3} \log^3(\tfrac{x}{u}) \right\} d\mathscr{B}_f(u) +O(\log^3 x)
  \\
   & = c_f \int_{1}^x \left\{ \frac{1}{2} \log^2(\tfrac{x}{u}) \log u + \frac{1}{3} \log^3(\tfrac{x}{u}) \right\} \frac{du}{u} +O(\log^3 x)
  \\
  & = c_f \left\{ \frac{1}{24} + \frac{1}{12} \right\} \log^{4}x+ O(\log^3 x)
  \\
&= \frac{1}{8} c_f  \log^{4}x+ O(\log^3 x).
\end{align*}
By combining estimates, we prove \eqref{first}.

\smallskip

The proof of \eqref{second} is similar, so we simply sketch the details. By again applying Lemma \ref{RSBound4} and then removing the prime powers, it can be shown that the sum on the left-hand side of \eqref{second} equals  
\begin{align*}
\sum_{p\leq x} \frac{\Lambda_{f}(p)\overline{\lambda_{f}(p)}}{p} \sum_{n\leq x/p} \frac{|\lambda_{f}(n)|^{2}}{n}  \log\Big(\frac{x^{2}}{n}\Big) \log\Big(\frac{x^{2}}{pn}\Big)   + O(\log^3 x). 
\end{align*}
By partial summation, using \eqref{eq:ransel} and  \eqref{eq:prime1}, and the identity $\Lambda_{f}(p)=\lambda_{f}(p)\log p$ this equals
\begin{align*}
 c_f \sum_{p\leq x} &\frac{|\lambda_{f}(p)|^2 \log p}{p}  \int_{1}^{\frac{x}{p}}  \log\Big(\frac{x^{2}}{t}\Big) \log\Big(\frac{x^{2}}{pt}\Big) \frac{dt}{t} + O(\log^3x )
 \\
 &= c_f \sum_{p\leq x} \frac{|\lambda_{f}(p)|^2 \log p}{p}  \int_{x}^{\frac{x^2}{p}}  \log(pu) \log(u)  \frac{du}{u} + O(\log^3x )
 \\
 &= c_f \sum_{p\leq x} \frac{|\lambda_{f}(p)|^2 \log p}{p} \left\{ \frac{1}{3} \log^3\Big(\frac{x^2}{p}\Big)+\frac{1}{2} \log^2\Big(\frac{x^2}{p}\Big)\log p- \Big(\frac{\log p}{2}+\frac{\log x}{3} \Big) \log^2x \right\} 
 \\
 &\quad \quad + O(\log^3x )
 \\
 &= c_f \int_1^x \left\{ \frac{1}{3} \log^3\Big(\frac{x^2}{v}\Big)+\frac{1}{2} \log^2\Big(\frac{x^2}{v}\Big)\log v- \Big(\frac{\log v}{2}+\frac{\log x}{3} \Big) \log^2x \right\} \frac{dv}{v}
 \\
 &\quad \quad  + O(\log^3x )
 \\
& = c_f \left\{ \frac{5}{4} + \frac{11}{24} - \Big(\frac{1}{4} + \frac{1}{3}\Big) \right\}\log^{4}x+ O(\log^3 x)
\\
& = \frac{9}{8} c_f \log^4 x  + O(\log^3x ).
\end{align*}
This establishes \eqref{second}, and thus completes the proof of Lemma \ref{lambda_prime}.
\end{proof}

The final lemma of this section deals with the order of poles arising in certain Dirichlet series which occur in the error term in Proposition \ref{dmva}. 
\begin{lemma} \label{pole}
For $f\in H_k(q,\chi)$ and $X \ge 2$,  the inequalities
\begin{align}
    \sum_{n=1}^\infty \frac{\big|(\Lambda_f\!*\!\lambda_f)(n)\big|^2}{n^\sigma} & \ll \frac{1}{(\sigma-1)^3},  \label{eq:pole1} \\
    \sum_{n=1}^\infty \frac{\big|(\Lambda_f\!*\!\alpha_f)(n)\big|^2}{n^\sigma} & \ll \frac{1}{(\sigma-1)^5}, \ \text{ and}  \label{eq:pole2} \\
     \sum_{n=1}^\infty \frac{\big|(\Lambda_f\!*\! \beta_{f,X})(n)\big|^2}{n^\sigma}  & \ll  \frac{ (\log X)^2}{(\sigma-1)^3}  +
      \frac{1}{(\sigma-1)^5} \label{eq:pole3}
\end{align}
hold uniformly as $\sigma\rightarrow 1^+$.
\end{lemma}

\begin{proof}
We begin by proving the estimates in \eqref{eq:pole1} and \eqref{eq:pole2}; the estimate in \eqref{eq:pole3} will be deduced as a consequence of these results.
For $f \in H_k(q,\chi)$ and $z$  a complex variable we define the arithmetic function
\[
  g_{f,z}(n) =\lambda_f(n) n^{-z}. 
\]
Note that 
\begin{equation}
  \label{arithidentity1}
 \lambda_f(n) = g_{f,0}(n)  \quad \text{ and } \quad \alpha_f(n) = -\lambda_f(n)\log n = \frac{d}{dz} \big[ g_{f,z}(n) \big]\Big|_{z=0}. 
\end{equation}
and
\begin{equation}
  \label{arithidentity2}
 \Lambda_f(n) = -(\mu_f\!*\!\alpha_f)(n) = -\frac{d}{dz} \big[ (\mu_f\!*\!g_{f,z})(n) \big]\Big|_{z=0}
\end{equation}
where  $\mu_f(\cdot)$ is defined by the generating series
\begin{equation}
  \label{eq:Lminus1}
 \frac{1}{L(s,f)} = \sum_{n=1}^\infty \frac{\mu_f(n)}{n^s}
\end{equation}
which converges absolutely for $\Re(s)>1$. 
Let $z_1,z_2,w_1,w_2$ be complex variables,  assume that 
\[
   \max_{i,j \in \{1,2\}}(|z_i|,|w_j|) \le \varepsilon
\]
for some small $\varepsilon>0$, and define 
$$ G_f(s) := G_f(s;w_1,w_2,z_1,z_2)= \sum_{n=1}^\infty \frac{(\mu_f\!*\!g_{f,w_1}\!*\!g_{f,z_1})(n)(\mu_{\overline{f}}\!*\!g_{\overline{f},w_2}\!*\!g_{\overline{f},z_2})(n)}{n^s}.$$
It follows from \eqref{arithidentity1} and \eqref{arithidentity2} that, for $\Re(s) > 1$,
\[
     \sum_{n=1}^\infty \frac{\big|(\Lambda_f\!*\!\lambda_f)(n)\big|^2}{n^s}  =
      \frac{d}{dw_1}\frac{d}{dw_2} G_f(s;w_1,w_2,0,0) \Bigg|_{w_1=w_2=0}
\]
and that
$$ \sum_{n=1}^\infty \frac{\big|(\Lambda_f\!*\!\alpha_f)(n)\big|^2}{n^s} = \frac{d}{dw_1}\frac{d}{dw_2}\frac{d}{dz_1}\frac{d}{dz_2} 
G_f(s;w_1,w_2,z_1,z_2) \Bigg|_{\substack{w_1=w_2=0 \\ z_1=z_2=0}}.$$
By the multiplicativity of the coefficients of $G_f(s)$, it follows that
\[
   G_f(s)=\prod_{p\text{ prime}} \! \left(1+ \frac{(\mu_f\!*\!g_{f,w_1}\!*\!g_{f,z_1})(p)(\mu_{\overline{f}}\!*\!g_{\overline{f},w_2}\!*\!g_{\overline{f},z_2})(p)}{p^s} + \cdots \right).
\]   
We now re-write this Euler product in terms of the Rankin-Selberg type convolution Dirichlet series 
\[
L_f(s):=L(s,f\!\times \!\bar{f}) = \sum_{n=1}^{\infty} \frac{ |\lambda_f(n)|^2}{ n^{s}}.
\] 
 Since $\mu_f(1)=1$, $\mu_f(p)=-\lambda_f(p)$, and $g_{f,z}(1)=1$ for each prime $p$ and every complex number $z$, it follows that
$
(\mu_f\!*\!g_{f,w}\!*\!g_{f,z})(p) 
= \lambda_f(p)\left( p^{-z}+p^{-w}-1\right)$
and thus for each prime $p$
$$ (\mu_f\!*\!g_{f,w_1}\!*\!g_{f,z_1})(p)(\mu_{\overline{f}}\!*\!g_{\overline{f},w_2}\!*\!g_{\overline{f},z_2})(p) = |\lambda_f(p)|^2\left( 1\!-\!p^{-w_1}\!-\!p^{-z_1}\right)\left( 1\!-\!p^{-w_2}\!-\! p^{-z_2}\right).$$
By \eqref{eq:Lminus1}, we have $|\mu_f(n)| \le d(n)$. Since $|\lambda_f(n)| \le d(n)$, it follows that 
\[
|g_{f,w}(n)| \le d(n) n^{-\Re(w)} \le d(n) n^{\varepsilon}
\]
and thus 
\[
     |(\mu_f\!*\!g_{f,w_i}\!*\!g_{f,z_i})(n)|  \le n^{2 \varepsilon} (d*d*d)(n) \ll n^{3 \varepsilon}. 
\]
Therefore
\begin{equation*}
\begin{split}
G_f(s) &=\prod_{p}  \left(1+ \frac{|\lambda_f(p)|^2\left( 1\!-\!p^{-w_1}\!-\!p^{-z_1}\right)\left( 1\!-\!p^{-w_2}\!-\! p^{-z_2}\right)}{p^s} + O ( p^{-2 \Re(s)+12 \varepsilon} )  \right)
\\
&=\prod_{p}  \Big(1\!+\!\frac{|\lambda_f(p)|^2}{p^s}\Big( 1\!-\! \frac{1}{p^{w_1}}\!-\!\frac{1}{p^{w_2}}\!-\!\frac{1}{p^{z_1}}\!-\!\frac{1}{p^{z_2}} +\cdots
\\
& \quad \quad \quad  \quad\quad  \quad \quad\quad\quad \quad  \quad \cdots\!+\!\frac{1}{p^{w_1+w_2}}\!+\!\frac{1}{p^{w_1+z_2}}\!+\!\frac{1}{p^{z_1+w_2}}\!+\!\frac{1}{p^{z_1+z_2}}\Big)\! 
+ O ( p^{-2 \Re(s)+12 \varepsilon} ) \Big)
\\
&= \frac{L_f(s)L_f(s\!+\!w_1\!+\!w_2)L_f(s\!+\!w_1\!+\!z_2)L_f(s\!+\! z_1\!+\!w_2) L_f(s\!+\!z_1\!+\!z_2)}{L_f(s\!+\!w_1)L_f(s\!+\!w_2)L_f(s\!+\!z_1)L_f(s\!+\!z_2) } \Pi_f(s)
\end{split}
\end{equation*}
where $\Pi_f(s)=\Pi_f(s;w_1,w_2,z_1,z_2)$ is an Euler product which converges absolutely in a half-plane containing the line $\Re(s)=1$ so long as 
 $\varepsilon \le  \tfrac{1}{100},$
say. Since $ L_f(s) \sim c_f(s\!-\!1)^{-1}$ 
near $s=1$, (this follows from \eqref{eq:ransel} of Lemma \ref{sums}) it follows that
$$ G_f(s) \sim  \frac{c_f (s\!+\!w_1\!-\!1)(s\!+\!w_2\!-\!1) (s\!+\!z_1\!-\!1)(s\!+\!z_2\!-\!1)\Pi_f(s;z_1,z_2,w_1,w_2)}{(s\!-\!1)(s\!+\!w_1\!+\!w_2\!-\!1)(s\!+\!w_1\!+\!z_2\!-\!1) (s\!+\! z_1\!+\!w_2\!-\!1) (s\!+\!z_1\!+\!z_2\!-\!1)} $$
when $s$ is near 1 and $w_1, w_2, z_1,$ and $z_2$ are all near 0. Introducing the differential operators
\[
  D_1(\cdot) = \left. \frac{d}{dw_1} \frac{d}{dw_2}\Big(\cdot\Big) \right|_{ w_1=w_2=0 }
\quad \text{ and } \quad
  D_2(\cdot) = \left. \frac{d}{dz_1} \frac{d}{dz_2} \frac{d}{dw_1} \frac{d}{dw_2}\Big(\cdot\Big) \right|_{ \substack{ w_1=w_2=0 \\  z_1=z_2=0} },
\]
it can be shown that
\begin{align*}
  D_1\!\left(  \frac{(s\!+\!w_1\!-\!1)(s\!+\!w_2\!-\!1)}{(s\!+\!w_1\!+\!w_2\!-\!1)(s\!+\!w_1\!-\!1) (s\!+\!w_2\!-\!1) } \right)\!=\! \frac{2}{(s\!-\!1)^3}
\end{align*}
and that
\begin{align*}
  D_2\!\left(  \frac{(s\!+\!w_1\!-\!1)(s\!+\!w_2\!-\!1) (s\!+\!z_1\!-\!1)(s\!+\!z_2\!-\!1)}{(s\!-\!1)(s\!+\!w_1\!+\!w_2\!-\!1)(s\!+\!w_1\!+\!z_2\!-\!1) (s\!+\! z_1\!+\!w_2\!-\!1) (s\!+\!z_1\!+\!z_2\!-\!1)} \right)\!=\! \frac{7}{(s\!-\!1)^5}. 
\end{align*}
Thus, near $s=1$, we have shown that
$$ \sum_{n=1}^\infty \frac{\big|(\Lambda_f\!*\!\lambda_f)(n)\big|^2}{n^s}  \sim \frac{2 c_f}{(s\!-\!1)^3}\Pi_f(1;0,0,0,0)$$
and
$$ \sum_{n=1}^\infty \frac{\big|(\Lambda_f\!*\!\alpha_f)(n)\big|^2}{n^s}  \sim \frac{7 c_f}{(s\!-\!1)^5}\Pi_f(1;0,0,0,0)$$
from which \eqref{eq:pole1} and \eqref{eq:pole2} now follow.  

\smallskip

It remains to establish the estimate in \eqref{eq:pole3}. Noting that $\beta_{f,X}(n) = - 2(\log X) \lambda_f(n) - \alpha_f(n)$,  it follows from the elementary inequality $|a+b|^2 \leq 2|a|^2 + 2|b|^2$ that
\begin{align*}
  \sum_{n=1}^\infty \frac{\big|(\Lambda_f\!*\!\beta_{f,X})(n)\big|^2}{n^{\sigma}}
  & \le 8 (\log X)^2  \sum_{n=1}^\infty \frac{\big|(\Lambda_f\!*\!\lambda_f)(n)\big|^2}{n^{\sigma}} + 2 \sum_{n=1}^\infty \frac{\big|(\Lambda_f\!*\!\alpha_f)(n)\big|^2}{n^{\sigma}} \\
  & \ll (\log X)^{2}  (\sigma-1)^{-3} + (\sigma-1)^{-5}
\end{align*}
using \eqref{eq:pole1} and \eqref{eq:pole2}. This completes the proof of the lemma.
\end{proof}

\section{Proofs of Proposition \ref{mainterms} and Proposition \ref{errordmv} } 

In this section, we deduce Proposition \ref{mainterms} and Proposition \ref{errordmv} from
Proposition \ref{dmva} and Proposition \ref{dmvbd}  in conjunction with the various results on sums of the arithmetic functions $\lambda_f(n)$ 
and $\Lambda_f(n)$ that were established in \textsection 5. Combined with the analysis in \textsection 1.3, this completes the proof of Theorem \ref{2nd_moment}.

\smallskip

\begin{proof}[Proof of Proposition \ref{mainterms}]
Let $T$ be large. We first establish the estimate
\begin{equation}\label{1stgoal}
 \sum_{T<\gamma_{\!_f}\leq 2T} \Bigg|\sum_{n\leq X} \frac{\alpha_f(n)}{n^{\rho_{\!_f}}} \Bigg|^2 =  \frac{5 }{24\pi} c_f  T  \log^4 X + O\Big( T (\log T)^{4-2\delta} \Big).
 \end{equation}
 By \eqref{eq:defncoeffs}, \eqref{eq:absval}, and \eqref{eq:ransel},  we have 
 \begin{align*}
   \sum_{n \le x}  |\alpha_f(n)|  \ll x (\log xT) (\log x)^{-\delta}
 \end{align*}
and
 \begin{align*}
  \sum_{n \le x} |\alpha_f(n)|^2  \ll x (\log xT)^2 \ll 
 \end{align*}
for $x \ll T$. Thus, conditions \eqref{eq:coefficientscondition1} and \eqref{eq:coefficientscondition2} are satisfied. 
We now apply Proposition \ref{dmva} with the choices $a(n) =\alpha_f(n)$, $Y=X$, and $\eta = \delta$ where $\delta$ is given by \eqref{eq:delta}.
It follows that
\begin{equation}\label{1ststep}
   \sum_{T<\gamma_{\!_f}\leq 2T} \Bigg|\sum_{n\leq X} \frac{\alpha_f(n)}{n^{\rho_{\!_f}}} \Bigg|^2   = \frac{c_f}{\pi} \left( \frac{1}{3}-\frac{1}{8} \right) T  \log  ^4X 
    + O\Big(  T (\log T)^{4-2 \delta} + T (\log T)^{\frac{7}{2}} \Big)
\end{equation}
since, by Lemma \ref{rankinselbergpartialsum}, Lemma \ref{lambda_prime}, and Lemma \ref{pole}, we have 
$$ \sum_{n\leq X} \frac{|\alpha_f(n)|^{2}}{n} =\frac{1}{3}c_{f} \log^3 X + O\big(\log^2 T\big), $$
$$ \sum_{n \le X}  \frac{(\Lambda_f*\alpha_f )(n) \overline{\alpha_f(n)}}{n} =\frac{1}{8} c_f  \log^4 X + O\big(\log^3 T\big) , $$
and
$$  \sum_{n=1}^{\infty} \frac{|(\Lambda_f*\alpha_f)(n)|^2}{n^{2c-1}}  \ll \frac{1}{(2c{-}1)^{5}} \ll \log^5 T.$$
Noting that  $\frac{7}{2} < 4-2 \delta$ we obtain \eqref{1stgoal} from \eqref{1ststep}. 

\smallskip

The estimate
\begin{equation}\label{2ndgoal}
\sum_{T<\gamma_{\!_f}\leq 2T} \Bigg|\sum_{n\leq X} \frac{\beta_{f,X}(n)}{n^{\rho_{\!_f}}} \Bigg|^2 = \frac{29 }{24\pi} c_f  T  \log^4 X + O\Big( T (\log T)^{4-2\delta} \Big)
\end{equation}
is similar.    By \eqref{eq:defncoeffs}, \eqref{eq:absval}, and \eqref{eq:ransel},  we have
 \[
   \sum_{n \le x}  |\beta_{f,X}(n)| \ll x (\log xT) (\log x)^{-\delta} \quad \text{ and } \quad \sum_{n \le x} |\beta_{f,X}(n)|^2 \ll x (\log xT)^2
 \]
 for $x \ll T$.
 Thus we may apply  Proposition \ref{dmva} with $a(n) = \beta_{f,X}(n)$, $Y=X$, and $\eta = \delta$ to obtain
\begin{equation}\label{2ndstep}
    \sum_{T<\gamma_{\!_f}\leq 2T} \Bigg|\sum_{n\leq X} \frac{\beta_{f,X}(n)}{n^{\rho_{\!_f}}} \Bigg|^2 = \frac{c_f}{\pi} \left( \frac{7}{3}-\frac{9}{8} \right) T \log^4 X 
    + O\Big(  T (\log T)^{4-2 \delta} + T (\log T)^{\frac{7}{2}} \Big)
 \end{equation}
since, by Lemma \ref{rankinselbergpartialsum}, Lemma \ref{lambda_prime}, and Lemma \ref{pole}, we have
$$  \sum_{n\leq X} \frac{|\beta_{f,X}(n)|^{2}}{n}  =\frac{7}{3}c_{f}  \log^3 X + O\big(\log^{2} T\big), $$
$$ \sum_{n \le X}  \frac{(\Lambda_f* \beta_{f,X} )(n) \overline{\beta_{f,X}(n)}}{n}  =\frac{9}{8} c_f  \log^4 X + O\big(\log^3 T\big),$$
and
$$ \sum_{n=1}^{\infty} \frac{|(\Lambda_f*\beta_{f,X})(n)|^2}{n^{2c-1}}  \ll \frac{\log^2 T}{(2c{-}1)^{3}} + \frac{1}{(2c{-}1)^{5}} \ll \log^5 T.$$
Here, as was the case in the previous estimate, we deduce \eqref{2ndgoal} from \eqref{2ndstep} by using the fact that $\delta \leq \frac{1}{4}$. 
\end{proof}

\begin{proof}[Proof of Proposition \ref{errordmv}]
Using the notation of Lemma \ref{approx}, we will prove that
\[
\sum_{T<\gamma_{\!_f}\leq 2T} \big|\mathcal{E}_j \big(\rho_{\!_f}, f) \big|^2 = O\left( \frac{T \log^4 T}{\log\log T}\right) 
\]
for each $j=1,\ldots,5$. From these five estimates, Proposition \ref{errordmv} then follows in an obvious manner. 

\smallskip

In order to obtain this result, we shall apply Proposition \ref{dmvbd} to each of the above five sums. 
Recalling the definitions in \eqref{eq:Esf}, we write $\mathcal{E}_j(\frac{1}{2}+it,f)= \sum_{n=1}^{\infty} b_j(n) n^{-it}$ for $j=1,2,3$ 
with the coefficients
\[
  b_j(n)  = \begin{cases}
                       \alpha_f(n) (e^{-n/X}{-}1) n^{-\frac{1}{2}} I_{[1,X]}(n), & \text{ if } j=1, \\
                     \alpha_f(n) e^{-n/X} n^{-\frac{1}{2}} I_{(X,\infty]}(n), & \text{ if } j=2, \text{ or} \\
                       \lambda_{f}(n) n^{-\frac{1}{2}} I_{[1,X]}(n), &  \text{ if } j=3,
                     \end{cases}
\]
respectively. Here $I_{J}(t)$ is the indicator function of the interval $J$. 
With these choices of coefficients, using Lemma \ref{rankinselbergpartialsum},  Proposition \ref{sums}, and the inequality $0\le 1{-} e^{-\frac{n}{X}} \leq \frac{n}{X}$, it is not hard to show that
$$ S_{1}  \ll T \log^2 T, \quad  S_2  \ll T \log^{4} T, \quad \text{and}  \quad S_3 \ll  \sqrt{T} (\log T)^{1-\delta}$$
where $S_1$, $S_2$, and $S_3$ are the requisite quantities in upper bound in Proposition \ref{dmvbd}. (In fact, when $j=3$, one can do a better by a factor of $\log T$ in each of the three estimates.) Thus, when $j=1, 2,$ or $3$ we deduce that
\begin{equation}\label{L123}
\sum_{T<\gamma_{\!_f}\leq 2T} \big|\mathcal{E}_j \big(\rho_{\!_f}, f) \big|^2  \ll  \frac{\log T}{\log \log T} \Big( \sqrt{S_1 S_2} + S_{3}^2\Big) + S_1 \log T 
  \ll \frac{ T \log^4 T}{\log \log T}.
\end{equation}

\smallskip

It remains to consider the cases $j=4$ and $j=5$. Recall that
\begin{equation}
  \begin{split}
  \label{eq:E4sf}
  &\mathcal{E}_4 \big(s, f) 
  \\
  &\quad = -\frac{  \varepsilon_{_f}}{2\pi i} \int_{-\frac{3}{4}-i\infty}^{-\frac{3}{4}+i\infty} \Gamma(w) X^w  \left(  
  \psi_{\!_f}'(s{+}w)\sum_{n>X} \frac{\lambda_{\bar f}(n)}{n^{1-s-w}}
  +   \psi_{\!_f}(s {+} w) \sum_{n>X} \frac{\lambda_{\bar f}(n) (\log n)}{n^{1-s-w}}
  \right) \ \! dw
  \end{split}
\end{equation}
and that 
\begin{equation}
\begin{split}
 \label{eq:E5sf}
   &\mathcal{E}_5 \big(s, f) 
   \\
   &\quad = -\frac{ \varepsilon_{_f}}{2\pi i} \int_{\frac{1}{4}-i\infty}^{\frac{1}{4}+i\infty} \Gamma(w) X^w 
   \left(  
  \psi_{\!_f}'(s{+}w)\sum_{n \le X} \frac{\lambda_{\bar f}(n)}{n^{1-s-w}}
  +   \psi_{\!_f}(s{+}w) \sum_{n \le X} \frac{\lambda_{\bar f}(n) (\log n)}{n^{1-s-w}}
  \right)
   \ \! dw.
   \end{split}
\end{equation}
We now estimate the tail of these integrals, allowing us to replace the right-hand sides of \eqref{eq:E4sf} and \eqref{eq:E5sf} by integrals with a short range of integration plus a small error term. In particular, the exponential decay of the gamma function allows us to restrict the integrals to the range $|\Re(w)| \leq \log^2T.$

\smallskip

For $s=\frac{1}{2} + i \gamma_f$ and $w = u + iv$ where $u=-\frac{3}{4}$ or $\frac{1}{4}$, Stirling's formula for the gamma function implies the uniform estimates  
\begin{equation}\label{11}  
|\Gamma(w)|  = 
\sqrt{2 \pi} |v|^{u-\frac{1}{2}} e^{-\frac{\pi}{2}|v|} \left(1+ O\Big(\tfrac{1}{|v|{+}1}\Big)\right), \qquad  \left| \frac{ \psi_{\!_f}'}{ \psi_{\!_f}}(s{+}w) \right|   \ll \log\Big(|v{+}\gamma_f|+2\Big) ,
\end{equation}
and
\begin{equation}
\label{33} 
 | \psi_{\!_f}(s{+}w)|   \ll \begin{cases}
    |v+\gamma_f|^{-2u}, &  \text{ if } |v+\gamma_f| \ge 1, \\
    1, & \text{ otherwise}. 
    \end{cases}
 \end{equation}  
We shall also make use of the trivial estimates
\begin{equation}\label{44}
     \Bigg| \sum_{n>X} \frac{\lambda_{\bar f}(n) (\log n)^j}{n^{\frac{7}{4}-\rho_{\!_f}-iv}} \Bigg| \ll T^{-\frac{1}{4}+\varepsilon}
\quad \text{and} \quad
     \Bigg| \sum_{n \le X} \frac{\lambda_{\bar f}(n) (\log n)^j}{n^{\frac{3}{4}-\rho_{\!_f}-iv}} \Bigg|  \ll T^{\frac{3}{4}+\varepsilon}
 \end{equation}
for every $\varepsilon>0$ and $j=0,1$, which follow from the $\text{RH}_f$ and the bound $|\lambda_f(n)|\le d(n)$. 
Using the estimates in \eqref{11},
\eqref{33}, and \eqref{44}, it follows that the tails of the integrals in \eqref{eq:E4sf} and \eqref{eq:E5sf} are bounded by a quantity which is
\begin{equation} \label{tail}
 \ll  T^{c_1} \int_{|v| \ge \log^2 T} |v|^{c_2} e^{-\frac{\pi}{2} |v|} \ \! dv  \ll T^{c_1}  \int_{|v| \ge \log^2 T}  e^{- \frac{\pi}{4} |v|} \ \! dv
   \ll T^{c_1} e^{- \frac{\pi}{4}(\log T)^2} \ll 1
\end{equation}
for some positive absolute constants $c_1$ and $c_2$. 

We now focus on bounding the parts of the integrals in \eqref{eq:E4sf} and \eqref{eq:E5sf} with $|\Re(w)| \le \log^2 T$. 
It is convenient to define
\begin{equation} 
  \label{eq:CjDjdefn}
\mathcal{C}_{j,v}(t) =  \sum_{n>X} \frac{\lambda_{\bar f}(n)(\log n)^j}{n^{\frac{5}{4} -iv}} n^{it}  \quad  \text{and} \quad   
\mathcal{D}_{j,v}(t) =  \sum_{n \le X} \frac{\lambda_{\bar f}(n)(\log n)^j}{n^{\frac{1}{4} -iv}} n^{it} 
\end{equation}
with $j=0,1$. Letting $s=\rho_{\!_f}=\frac{1}{2}+i\gamma_{\!_f}$ and $w=-\frac{3}{4}+iv $  in \eqref{eq:E4sf}
and then  applying  \eqref{11}, \eqref{33}, \eqref{44}, and \eqref{tail}, we obtain 
\begin{align*}
&  |\mathcal{E}_4 \big(\rho_{\!_f}, f)|  \ll 1 + T^{\frac{3}{4}} \int_{-\log^2 T}^{\log^2 T} e^{-\frac{\pi}{2}|v|} 
   \Big( | \mathcal{C}_{0,v}(\gamma_f)| \log T+  |\mathcal{C}_{1,v}(\gamma_f)| \Big) \ \! dv.
\end{align*}
Similarly, we find that
\begin{align*}
&  |\mathcal{E}_5 \big(\rho_{\!_f}, f)|  \ll 1 + T^{-\frac{1}{4}} \int_{-\log^2 T}^{\log^2 T} e^{-\frac{\pi}{2}|v|} 
  \Big( | \mathcal{D}_{0,v}(\gamma_f)| \log T+  |\mathcal{D}_{1,v}(\gamma_f)| \Big) \ \! dv.
\end{align*}
We now sum these estimates over the zeros $\rho_{\!_f}=\frac{1}{2}+i\gamma_{_f} $ satisfying $T<\gamma_{_f} \le 2T.$

\smallskip

We first estimate the sum involving the mean-square of $\mathcal{E}_4 \big(\rho_{\!_f}, f)$. Using the inequality $|a{+}b|^2 \leq 2 |a|^2 {+} 2|b|^2$, followed by an application of Cauchy's inequality, then interchanging the sum and the integral, and then using the first inequality again, we have
\begin{equation} \label{E4}
 \begin{split} 
&  \sum_{T<\gamma_{\!_f}\leq 2T}   |\mathcal{E}_4 \big(\rho_{\!_f}, f)|^2 
  \\
   &\quad \ll T \log T + T^{3/2}  \sum_{T<\gamma_{\!_f}\leq 2T}  \left(  \int_{-\log^2 T}^{\log^2 T}  e^{-\frac{\pi}{2} |v|}   \Big( | \mathcal{C}_{0,v}(\gamma_f)| \log T+  |\mathcal{C}_{1,v}(\gamma_f)| \Big) \ \! dv \right)^2
      \\
   &\quad \ll T \log T + T^{3/2}   \sum_{T<\gamma_{\!_f}\leq 2T} \int_{-\log^2 T}^{\log^2 T}  e^{-\frac{\pi}{2} |v|}   \Big( | \mathcal{C}_{0,v}(\gamma_f)| \log T+  |\mathcal{C}_{1,v}(\gamma_f)| \Big)^2  \ \! dv  
      \\
   &\quad \ll T \log T + T^{3/2}    \int_{-\log^2 T}^{\log^2 T}  e^{-\frac{\pi}{2} |v|}  \left( \sum_{T<\gamma_{\!_f}\leq 2T}  \Big( | \mathcal{C}_{0,v}(\gamma_f)|^2 \log^2 T+  |\mathcal{C}_{1,v}(\gamma_f)|^2 \Big) \right) \ \! dv
 \end{split} 
\end{equation}
where, to derive the inequality in the penultimate line, we used the fact that 
\begin{equation} \label{expint}
\int_{-\log^2 T}^{\log^2 T} e^{-\frac{\pi}{2} |v|} \ \! dv \leq   \int_{-\infty}^{\infty} e^{-\frac{\pi}{2} |v|} \ \! dv 
\ll 1.
\end{equation}
To estimate the sums over zeros involving  $| \mathcal{C}_{0,v}(\gamma_f)|^2$ and  $| \mathcal{C}_{1,v}(\gamma_f)|^2$, we apply Proposition \ref{dmvbd} with the coefficients
\[
b(n)=\lambda_f(n) (\log n)^j n^{-5/4+iv}I_{(X,\infty)}(n)
\] 
where $j=0,1$. In the notation of Proposition \ref{dmvbd}, using \eqref{eq:ransel}, \eqref{eq:absval}, and partial summation, we have
$$ S_1    \ll  \sum_{n > X}  \frac{|\lambda_f(n)|^2 (\log n)^{2j} }{n^{\frac{3}{2}}}  \ll \frac{(\log T)^{2j}}{T^{\frac{1}{2}}}, $$
$$  S_2  \ll  \sum_{n > X}  \frac{|\lambda_f(n)|^2 (\log n)^{2j+2} }{n^{\frac{3}{2}}} \ll \frac{(\log T)^{2j+2}}{T^{\frac{1}{2}}}, $$
and 
$$ S_3   \ll  \sum_{n > X} \frac{|\lambda_f(n)| (\log n)^j}{n^{\frac{5}{4}}}  \ll \frac{(\log T)^{j-\delta}}{T^{\frac{1}{4}}} $$
where $\delta\ge \frac{1}{18}$ is the constant defined \eqref{eq:absval}. Hence
\begin{align*}
     \sum_{T<\gamma_{\!_f}\leq 2T}   | \mathcal{C}_{j,v}(\gamma_{f})|^2 \ll 
       \frac{\log T}{\log  \log T} \Big( \sqrt{S_1 S_2} + S_{3}^2\Big) + S_1 \log T
    \ll \frac{(\log T)^{2j+2}}{T^{\frac{1}{2}} \log \log T}
\end{align*}
for $j=0,1$. Inserting these estimates into the integral on the right-hand side of \eqref{E4} and again using \eqref{expint}, we find that
\begin{equation} \label{E4bound}
   \sum_{T<\gamma_{\!_f}\leq 2T}  |\mathcal{E}_4 \big(\rho_{\!_f}, f)|^2  \ll \frac{ T \log^4 T }{\log \log T}.
\end{equation}

\smallskip

A similar argument shows that 
 \begin{align*}
  & \sum_{T<\gamma_{\!_f}\leq 2T}  |\mathcal{E}_5 \big(\rho_{\!_f}, f)|^2 
  \\
  &\qquad \ll T\log T + T^{-\frac{1}{2}}  \!\!\!
     \int_{-\log^2 T}^{\log^2 T}  e^{-\frac{\pi}{2} |v|}  \left( \sum_{T<\gamma_{\!_f}\leq 2T}  \Big( | \mathcal{D}_{0,v}(\gamma_f)|^2 \log^2 T+  |\mathcal{D}_{1,v}(\gamma_f)|^2 \Big) \right) \ \! dv.
\end{align*}
Once again we apply  Proposition \ref{dmvbd} with the coefficients
\[
   b(n) = \frac{\lambda_f(n)(\log n)^j}{n^{\frac{1}{4} -iv}}I_{[1,X]}(n).
\]
This yields 
\[
  S_1 \ll  T \sum_{n \le X} \frac{|\lambda_f(n)|^2 (\log n)^{2j}}{\sqrt{n}}
  \ll T^{\frac{3}{2}} (\log T)^{2j},
\]
\[
 S_2 \ll T \sum_{n \le X} \frac{|\lambda_f(n)|^2 (\log n)^{2j+2}}{\sqrt{n}}
  \ll T^{\frac{3}{2}} (\log T)^{2j+2},
\]
and 
\[
   S_3 \ll  \sum_{n \le X} \frac{|\lambda_f(n)|(\log n)^j}{n^{\frac{1}{4}}} \ll  T^{\frac{3}{4}}(\log T)^j. 
\]
Consequently, 
\[
 \sum_{T<\gamma_{\!_f}\leq 2T}   | \mathcal{D}_{j,v}(\gamma_{f})|^2  \ll 
   \frac{\log T}{\log  \log T} \Big( \sqrt{S_1 S_2} + S_{3}^2\Big) + S_1 \log T  \ll 
  \frac{T^{\frac{3}{2}} (\log T)^{2j+2}}{\log \log T}
\]
for $j=0,1$. Whence
\begin{equation} \label{E5}
   \sum_{T<\gamma_{\!_f}\leq 2T}  |\mathcal{E}_5 \big(\rho_{\!_f}, f)|^2  \ll \frac{ T \log^4 T}{\log \log T}.
\end{equation}
Combining the estimates \eqref{L123}, \eqref{E4bound}, and \eqref{E5}, we complete the proof of Proposition \ref{errordmv}. 
\end{proof}

\smallskip

\section{Proof of Theorem \ref{shifted} }

In this section, we prove the following lemma on the distribution of values of $\log|L(s,f)|$ near the zeros of $L(s,f)$ and then use this result to derive Theorem \ref{shifted}.

\begin{lemma} \label{value_dist}
Assume $\text{RH}_f$ and let $T$ be large. For $w\in\mathbb{C}$ satisfying $|w|\le 1$ and $0\le \Re(w) \le (\log T)^{-1}$, let $\mathcal{N}(V;T,w)$ be the number of zeros $\rho_{_f}=\frac{1}{2}+i\gamma_{_f}$ satisfying $T<\gamma_{_f}\le 2T$ and
$\log|L(\rho_{_f}\!+\!w,f)|\ge V.$ Then the following inequalities for $\mathcal{N}(V;T,w)$ hold:
\begin{enumerate}
\item[\textup{(}i\textup{)}] For $10\sqrt{\log\log T} \le V \le \log\log T$, we have
\begin{equation*}
\mathcal{N}(V;T,w) \ll T\log T \frac{V}{\sqrt{\log\log T}} \exp\left( -\frac{V^2}{\log\log T} \left(1-\frac{8}{\log\log\log T} \right) \right)
\end{equation*}
\item[\textup{(}ii\textup{)}] For $ \log\log T\ \le V \le \frac{1}{4}\log\log T \log\log\log T$, we have
\begin{equation*}
\mathcal{N}(V;T,w) \ll T\log T \frac{V}{\sqrt{\log\log T}} \exp\left( -\frac{V^2}{\log\log T}\left(1-\frac{18V}{5 \log\log T\log\log\log T}  \right)^2 \right)
\end{equation*}
\item[\textup{(}iii\textup{)}] For $ V \ge \frac{1}{4}\log\log T \log\log\log T$, we have
\begin{equation*}
\mathcal{N}(V;T,w) \ll T\log T \exp\left(- \frac{1}{201} \, V \log V\right).
\end{equation*}
\end{enumerate}
\end{lemma}

\begin{proof} 
The proof of this lemma is a discrete version of Soundararajan's main theorem in \cite{S}, and is proved by combining the inequality for $\log|L(s,f)|$ in Lemma \ref{1st_inequality} with the mean-value estimates in Proposition \ref{dmv:primes}.  Throughout the proof, we assume that $T$ is large.

\smallskip

We define a parameter
\[
A = \left\{ \begin{array}{ll}
 \frac{1}{2}\log\log\log T, &\mbox{ if \ $10 \sqrt{\log \log T} \le V \le \log \log T$,} \\
  \frac{\log\log T}{2V}\log\log\log T, &\mbox{ if \ $ \log\log T < V \le \frac{\log \log T}{4}\log\log\log T$, }\\
 2, &\mbox{ if $V \ge \frac{\log \log T}{4}\log\log\log T$,}
       \end{array} \right.
\]
set $x=T^{A/V}$, and put $z=x^{1/\log \log T}$. 
Since $\mu_0 \ge \frac{1}{2}$, $x\le T^{1/2}$, and $0 \le \Re(w) \leq (\log T)^{-1}$, it follows that
\begin{equation}\label{zero_bound}
\frac{1}{2} \le \Re\big(\rho_{\!_f}\!+\!w \big) \leq \frac{1}{2}+\frac{1}{\log T} \le \frac{1}{2}+\frac{\mu_0}{\log x} 
\end{equation}
for any non-trivial zero $\rho_{\!_f} = \frac{1}{2}+i \gamma_{\!_f} $ of $L(s,f)$ with $T \le \gamma_{\!_f} \le 2T.$  Since $\mu_0 < \frac{3}{5}$, by choosing $x=\log T$  in Lemma \ref{1st_inequality} and using the bound $|\Lambda_f(n)|\le 2 \Lambda(n)$ to estimate the sum on the right-hand side of the inequality in the lemma, we see that
\[
\log|L(\sigma+it,f)| \le  \frac{8}{5}  \frac{\log T}{\log\log T} 
\]
for $\frac{1}{2} \leq \sigma \leq \frac{1}{2}+\frac{\mu_0}{\log x}$ and $T\le t \le 2T$. We may therefore assume that $V\le (8  \log T)/(5 \log \log T)$ as otherwise $\mathcal{N}(V;T,w)=0$.

\smallskip

Again since $\mu_0<3/5$, by equation \eqref{prime_inequality} in the remark following the proof of Lemma \ref{1st_inequality} and \eqref{zero_bound}, we have
\[
\log|L(\rho_{\!_f}\!+\!w,f)| \leq |\mathcal{S}_1(\rho_{\!_f})| + |\mathcal{S}_2(\rho_{\!_f})| + |\mathcal{S}_3(\rho_{\!_f})| + \frac{8V}{5A}
\]
for $T\le \gamma_{\!_f}\le 2T$ where
\[
\mathcal{S}_1(s)= \sum_{p\le z} \frac{\lambda_f(p)}{p^{s+\mu_0/\log x}} \frac{\log(x/p)}{\log x}, \quad \quad  \mathcal{S}_2(s) = \sum_{z<p\le x} \frac{\lambda_f(p)}{p^{s+\mu_0/\log x}} \frac{\log(x/p)}{\log x},  
\]
and
\[
\mathcal{S}_3(s)= \sum_{p\le \sqrt{x}} \frac{(\lambda_f(p^2)\!-\!\chi(p) )}{p^{2s+2\mu_0/\log x}} \frac{\log(\sqrt{x}/p)}{\log \sqrt{x}}.
\]
Therefore, if there is a zero $\rho_{\!_f}$ of $L(s,f)$ with $T< \gamma_{\!_f}\le 2T$ and $\log|L(\rho_{\!_f}+w,f)|>V$, then either
\begin{equation*} 
|\mathcal{S}_1(\rho_{\!_f})| \ge V \left(1-\frac{9}{5A}\right), \quad |\mathcal{S}_2(\rho_{\!_f})| \ge \frac{V}{10A}, \quad \text{or} \quad  |\mathcal{S}_3(\rho_{\!_f})| \ge \frac{V}{10A}.
\end{equation*}
For ease of notation, set $V_1 =V \left(1\!-\!\frac{9}{5A}\right)$ and $V_2=V_3 =\frac{V}{10A}.$  We now estimate the quantities
\[
   N_j(T;V) = \# \{ T < \gamma_f \le 2T \ | \ |\mathcal{S}_j(\rho_f)| \ge V_j \}
\]
for $j=1,2,3$ using the inequality 
\begin{equation}
   \label{Nibd}
   N_j(T;V)  \le V_{j}^{-2\ell} \sum_{ T < \gamma_f \le 2T}  |\mathcal{S}_j(\rho_f)|^{2\ell}.
\end{equation}
Here $\ell$ is a natural number to be chosen appropriately according to various cases. 

\smallskip

We first estimate $N_1(T;V)$.  By the first mean-value estimate in Proposition \ref{dmv:primes} and \eqref{eq:prime2} of Proposition \ref{sums}, we see that
\begin{eqnarray}\label{S1_bound}
\sum_{T < \gamma_{\!_f}\le 2T} \big|\mathcal{S}_1(\rho_{\!_f})\big|^{2\ell} \ll  \ell! \ \! T\log T \left(\sum_{p\leq z} \frac{|\lambda_f(p)|^2}{p} \right)^\ell \ll T\log T \sqrt{\ell} \left(\frac{\ell \log\log T}{e} \right)^\ell
\end{eqnarray}
for any integer $\ell$ with $z^\ell \leq T^{2/3}$.  By \eqref{Nibd}, we derive that
\begin{equation}
   \label{N1bd}
   N_1(T;V) \ll  T\log T \sqrt{\ell} \left(\frac{\ell \log\log T}{eV_1^2} \right)^\ell.
\end{equation}
We now choose $\ell$ based on the size of $V$. When $V \le (\log \log T)^2$ we take $\ell$ to be $\lfloor V_1^2/\log\log T \rfloor$, and when $V>(\log \log T)^2$ we take $\ell=\lfloor 10V \rfloor$ where $\lfloor x \rfloor$ denotes the greatest integer less than or equal to $x$. 
In either case, it may be verified that $z^\ell \leq T^{2/3}$ holds. 
 In the first case, by \eqref{N1bd}, we have 
\[
  N_1(T;V)  \ll   T\log T  \frac{V}{\sqrt{\log\log T}} \exp\left(-\frac{V_1^2}{\log\log T} \right)
\]
when $V \le (\log \log T)^2$. In the second case, the inequality in \eqref{N1bd} implies that
 \begin{align*}
  N_1(T;V) 
  & \ll T \log T  \exp  \Big( -4 V \log V \Big)
\end{align*}
since $\log V \ge 2 \log_3 T$ in the case when $V > (\log \log T)^2$. Combining both estimates, we deduce
\[
  N_1(T;V) 
\ll T\log T \left\{ \frac{V}{\sqrt{\log\log T}} \exp\left(-\frac{V_1^2}{\log\log T} \right) + \exp\Big(-4V\log V \Big) \right\}.
\]

\smallskip

Similarly, we estimate $N_2(T;V)$.  Again using the first mean-value estimate Proposition \ref{dmv:primes} and the bound $|\lambda_f(p)|\le 2$, we see that
\begin{eqnarray}\label{S2_bound}
\sum_{T < \gamma_{\!_f}\le 2T} \big|\mathcal{S}_2(\rho_{\!_f})\big|^{2\ell} \ll  \ell! \ \! T\log T \left(\sum_{z<p\leq x} \frac{|\lambda_f(p)|^2}{p} \right)^\ell \ll T\log T \left(4 \ell \log\log \log T \right)^\ell
\end{eqnarray}
for any integer positive integer $\ell \le (2V)/(3A)-1$. Choosing $\ell=\lfloor (2V)/(3A) \rfloor -1$, it follows from \eqref{Nibd} and \eqref{S2_bound} that
\[
 N_2(T;V) \ll T\log T  \exp\left(-\frac{V}{3A} \log V \right).
\]

\smallskip

It remains to estimate $N_3(T;V)$. By the second mean-value estimate in Proposition \ref{dmv:primes} and the bound $|\lambda_f(p^2)-\chi(p)|^2 \leq 4$, it follows that
\begin{equation}
   \label{S3_bound}
\sum_{T < \gamma_{\!_f}\le 2T} \big|\mathcal{S}_3(\rho_{\!_f})\big|^{2\ell} \ll \ell ! \ \!  T \log T \left(\sum_{p\le \sqrt{x}} \frac{4}{p^2} \right)^\ell \ll ( 2 \ell )^\ell \, T \log T 
\end{equation}
for any integer positive integer $\ell \le (2V)/(3A)-1$. Here we have used the fact that $\sum_{p} p^{-2} < \frac{1}{2}$. Note that this is smaller than the upper bound in \eqref{S2_bound}, so choosing $\ell=\lfloor (2V)/(3A) \rfloor -1$ it follows from \eqref{Nibd} and \eqref{S3_bound} that 
\[
 N_3(T;V) \ll T\log T  \exp\left(-\frac{V}{3A} \log V \right),
\]
as well. Combining our estimates for $N_1(T;V)$, $N_2(T;V)$, and $N_3(T;V)$, we deduce the inequality 
\[
    \mathcal{N}(V;T,w) \ll T\log T
    \Big( 
    \frac{V}{\sqrt{\log \log T}}  \exp\left(-\frac{V_1^2}{\log\log T} \right)   + \exp(-4V \log V) +  \exp\left(-\frac{V}{3A} \log V \right)
    \Big). 
\]
The lemma now follows from the definitions of $A$ and $V_1$ by considering separately each of the three indicated ranges of $V$.
\end{proof}

\smallskip

We now deduce Theorem \ref{shifted} as a consequence of the preceding lemma.

\begin{proof}[Proof of Theorem \ref{shifted}] We first establish the theorem when $0\le \Re(w) \le (\log T)^{-1}$ and then, using the functional equation for $L(s,f)$ and Stirling's formula for $\Gamma(s)$, extend the proof to all $w \in \mathbb{C}$ with $|w|\le 1$ and $|\Re(w)| \le (\log T)^{-1}$. Using the notation in Lemma \ref{value_dist},  we have
\begin{equation} \label{measure}
\begin{split}
\sum_{T<\gamma_{\!_f}\leq 2T} \big|L(\rho_{\!_f}\!+\! w,f)\big|^{2\ell} &= - \int_{-\infty}^{\infty} \exp(2 \ell V) \ \! d \mathcal{N}(V;T,w) 
\\
&= 2 \ell  \int_{-\infty}^{\infty} \exp(2 \ell V) \ \! \mathcal{N}(V;T,w) \ \! dV
\end{split}
\end{equation}
for any $\ell>0$ and $w \in \mathbb{C}$ with $|w|\le 1$ and $0<\Re(w) \le (\log T)^{-1}$.  Trivially, $ \mathcal{N}(V;T,w) \le N_{\!_f}(T) \ll T \log T$.
From the first two cases of Lemma \ref{value_dist}, we see that
\[
\mathcal{N}(V;T,w) \ll T(\log T)^{1+\varepsilon} \exp\left(\frac{-V^2}{\log\log T}\right) \quad \text{for } 3\leq V \leq 4 \ell \log\log T
\]
and, from the second two cases of Lemma \ref{value_dist}, that
\[
\mathcal{N}(V;T,w) \ll T(\log T)^{1+\varepsilon} \exp\left(-4 \ell V \right) \quad \text{for }  V > 4 \ell \log\log T.
\]
Using these three bounds for $\mathcal{N}(V;T,w)$ in the second integral on the right-hand side of \eqref{measure}, we have
\begin{equation*}
\begin{split}
   \sum_{T<\gamma_{\!_f}\leq 2T} \big|L(\rho_{\!_f}\!+\! w,f)\big|^{2\ell}    &\ll_{\ell} \
   T \log T  + T(\log T)^{1+\varepsilon}  \int_{3}^{4 \ell \log \log T}  \exp \Big( 2 \ell V - \frac{V^2}{\log \log T} \Big) dV \\
   & \quad \quad + T(\log T)^{1+\varepsilon}    \int_{4 \ell \log \log T}^{\infty} \exp\left(2 \ell V-4 \ell V \right)   dV.
\end{split}
\end{equation*}
After a little calculation, it follows that
\begin{equation}\label{w_right}
\sum_{T<\gamma_{\!_f}\leq 2T} \big|L(\rho_{\!_f}\!+\! w,f)\big|^{2\ell} \ll_{f, \ell,\varepsilon} T(\log T)^{\ell^2 + 1 +\varepsilon} 
\end{equation}
for any $\ell>0$ and $|w|\le 1$ with $0<\Re(w) \le (\log T)^{-1}$.

\smallskip

Now suppose that $|w|\le 1$ with $-(\log T)^{-1} \le \Re(w) <0.$ Then, assuming $\text{RH}_f$ and using Stirling's formula, we see that $|\psi_{\!_f}(\rho_{\!_f}+w)| \le C$ when $T<\gamma_{\!_f}\leq 2T$ for some absolute constant $C>0$ where $\psi_{\!_f}(s)$ corresponds to the function in the asymmetric form of the functional equation for $L(s,f)$ in \eqref{eq:func_eqn}. Thus, assuming $\text{RH}_f$, 
\[
 \big|L(\rho_{\!_f}\!+\! w,f)\big| = \big| \psi_{\!_f}(\rho_{\!_f}\!+\!w)L(1\!-\!\rho_{\!_f}\!-\! w,\bar{f})\big| \le C  \big|L(\rho_{\!_f}\!-\! \overline{w},f)\big| 
\]
when $T<\gamma_{\!_f}\leq 2T$. Here we have used that $1\!-\!\rho_{\!_f} = \overline{\rho_{\!_f}}$. It now follows from this inequality and \eqref{w_right} that 
\[
\sum_{T<\gamma_{\!_f}\leq 2T} \big|L(\rho_{\!_f}\!+\! w,f)\big|^{2\ell} \ll_{f,\ell,\varepsilon} T(\log T)^{\ell^2 + 1 +\varepsilon} 
\]
for any $\ell>0$ and $|w|\le 1$ with $|\Re(w)| \le (\log T)^{-1}$. Theorem \ref{shifted} now follows by summing this estimate over the dyadic intervals $(T/2,T], (T/4,T/2], (T/8,T/4], \ldots$.
\end{proof}

\section{Theorem \ref{shifted} implies Theorem \ref{2kth_moment}} We now deduce Theorem \ref{2kth_moment} from Theorem \ref{shifted} using the following lemma.

\begin{lemma} \label{Cauchy_integral_formula}
 Let $m \in \mathbb{N}$, $\ell \ge \frac{1}{2}$,  and  $0<R < \frac{1}{2}$. Then 
\begin{equation}
\sum_{0<\gamma_{_f}\le T} \big|L^{(m)}(\rho_{_f},f) \big|^{2\ell} \leq \Big(\frac{m!}{R^m} \Big)^{2\ell}\cdot \left[ \max_{|w|\leq R} \ \sum_{0<\gamma_{_f}\le T} \big|L(\rho_{_f}\!+\!w) \big|^{2\ell} dt \right].
\end{equation}
\end{lemma}

\begin{proof} The proof is similar to the proof of Lemma 4.5 in \cite{M2}. By Cauchy's integral formula, we have
\begin{equation}\label{Cauchy_int}
\sum_{0<\gamma_{_f}\le T}  \big|L^{(m)}(\rho_{_f},f) \big|^{2\ell}  = \Big(\frac{m!}{2\pi} \Big)^{2\ell} \cdot \sum_{0<\gamma_{_f}\le T} \left| \int_{\mathscr{C}_R} \frac{L(\rho_{_f}\!+\!w,f)}{w^{m+1}} \ \! dw \right|^{2\ell}  \\
\end{equation}
where $\mathscr{C}_R$ denotes the positively oriented circle in the complex plane centered at $0$ of radius $R$. The circumference of $\mathscr{C}_R$ is $2\pi R$ and $|w|=R$ for $w\in \mathscr{C}_R$, it follows from \eqref{Cauchy_int} that
\[
    \sum_{0<\gamma_{_f}\le T}\big|L^{(m)}(\rho_{_f}\!,f)\big|^{2\ell}  \le
     \Big(\frac{m!}{2\pi} \Big)^{2\ell} \cdot \frac{1}{R^{2\ell(m+1)} } \cdot \sum_{0<\gamma_{_f}\le T}    \left\{ \int_{\mathscr{C}_R} \big|L(\rho_{_f}\!+\!w,f)\big| \ \! |dw|  \right\}^{2\ell}.  
\]
By an application of H\"{o}lder's inequality, it follows that for $\ell > \frac{1}{2}$ 
\[
    \left\{ \int_{\mathscr{C}_R} \big|L(\rho_{_f}\!+\!w,f)\big| \ \! |dw|  \right\}^{2\ell}  
 \le  \left\{ \int_{\mathscr{C}_R} |dw| \right\}^{2\ell-1} \cdot \left\{   \int_{\mathscr{C}_R} \big|L(\rho_{_f}\!+\!w,f)\big|^{2\ell} \ \! |dw| 
 \right\}.
\]
Note that this inequality holds trivially in the case $\ell=\frac{1}{2}$.  Therefore, for $\ell \ge \frac{1}{2},$ we have
\begin{equation*}
\begin{split}
\sum_{0<\gamma_{_f}\le T}\big|L^{(m)}(\rho_{_f}\!,f)\big|^{2\ell} 
&\leq \Big(\frac{m!}{2\pi} \Big)^{2\ell} \cdot \frac{(2\pi R)^{2\ell-1}}{R^{2\ell(m+1)} } \cdot \sum_{0<\gamma_{_f}\le T} \left\{ \int_{\mathscr{C}_R} \big|L(\rho_{_f}\!+\!w,f)\big|^{2\ell} \ \! |dw| \right\}
\\
& = \Big(\frac{m!}{R^m} \Big)^{2\ell} \cdot \frac{1}{2\pi R} \cdot \int_{\mathscr{C}_R} \left\{ \sum_{0<\gamma_{_f}\le T}  \big|L(\rho_{_f}\!+\!w,f)\big|^{2\ell} \right\}  \ \! |dw| 
\\
&\leq \Big(\frac{m!}{R^m} \Big)^{2\ell}  \cdot \left[ \max_{|w|\leq R} \ \sum_{0<\gamma_{_f}\le T} \big|L(\rho_{_f}\!+\!w) \big|^{2\ell}  \right]
\end{split}
\end{equation*}
as claimed. This completes the proof of the lemma.
\end{proof}

\begin{proof}[Proof of Theorem \ref{2kth_moment}] Let $\ell \ge \frac{1}{2}$ and set $R=(\log T)^{-1}$.  Then it follows from Theorem \ref{shifted} and Lemma \ref{Cauchy_integral_formula} that
\[
\sum_{0<\gamma_{_f}\le T} \big|L^{(m)}(\rho_{_f},f) \big|^{2\ell}  \ll_{f,m,\ell,\varepsilon} T (\log T)^{\ell^2+2\ell m+1+\varepsilon}
\]
for any  $m\in \mathbb{N}$ and $\varepsilon>0$ arbitrary. This proves Theorem \ref{2kth_moment}.
\end{proof}

\bigskip

\noindent{\it Acknowledgements.} We thank Matt Young and David Farmer for some helpful comments, and Freydoon Shahidi for his correspondences concerning the estimate for  $\mathscr{D}_f(x)$ in Proposition \ref{sums}. The first author benefitted from conservations with Brian Conrey and Hung Bui during the early stages of this project. He thanks them for their encouragement.

\bibliographystyle{amsplain}

\end{document}